\documentclass[11pt,twoside,reqno,centertags,english]{amsart}
\usepackage[utf8]{inputenc}
\usepackage{babel}
\usepackage{lmodern}
\usepackage{bbding}
\usepackage{hyperref}
\usepackage{geometry}
\usepackage{comment}
\hypersetup{
	colorlinks,
	citecolor=blue,
	filecolor=blue,
	linkcolor=blue,
	urlcolor=blue
}
\def\be{\begin{equation}}
	\def\ee{\end{equation}}
\newcommand {\p} {\partial}
\usepackage{amsmath, amsfonts, amssymb}
\usepackage{amsthm}
\newtheorem{theorem}{Theorem}[section]
\newtheorem{lemma}[theorem]{Lemma}

\newtheorem{corollary}[theorem]{Corollary}
\newtheorem{proposition}[theorem]{Proposition}
\theoremstyle{remark}
\newtheorem{remark}[theorem]{Remark}

\theoremstyle{definition}
\newtheorem{definition}[theorem]{Definition}
\newtheorem{openproblem}[theorem]{Open Problem}
\newtheorem{assumption}[theorem]{Assumption}
\newtheorem{conv}[theorem]{Convention}
\newtheorem{setting}[theorem]{Setting}
\usepackage{mathtools}

\makeatletter
\DeclareRobustCommand\widecheck[1]{{\mathpalette\@widecheck{#1}}}
\def\@widecheck#1#2{%
	\setbox\z@\hbox{\m@th$#1#2$}%
	\setbox\tw@\hbox{\m@th$#1%
		\widehat{%
			\vrule\@width\z@\@height\ht\z@
			\vrule\@height\z@\@width\wd\z@}$}%
	\dp\tw@-\ht\z@
	\@tempdima\ht\z@ \advance\@tempdima2\ht\tw@ \divide\@tempdima\thr@@
	\setbox\tw@\hbox{%
		\raise\@tempdima\hbox{\scalebox{1}[-1]{\lower\@tempdima\box
				\tw@}}}%
	{\ooalign{\box\tw@ \cr \box\z@}}}
\makeatother

\DeclareMathOperator{\tr}{Tr}

\DeclareMathOperator{\di}{div}
\newcommand{\vertiii}[1]{{\left\vert\kern-0.25ex\left\vert\kern-0.25ex\left\vert #1 
		\right\vert\kern-0.25ex\right\vert\kern-0.25ex\right\vert}}
\newcommand{\bigk}{\mathcal K(E,B)}
\newcommand{\bigkp}{\mathcal K(E_+,B_+)}

\newcommand {\R} {\mathbb{R}}
\newcommand {\ms} {\mathcal M}
\newcommand {\lc} {\mathcal A}

\newcommand{\la}{\left\langle}
\newcommand{\ra}{\right\rangle}

\newcommand{\nn}{N\times N}

\numberwithin{equation}{section}
\newcommand{\Wei}{\mathfrak{H}}

\newcommand{\wei}{\mathfrak{h}}

\newcommand{\mut}{\mathcal{L}}
\newcommand{\mom}{v^{\#}}
\newcommand{\gamex}{\nu}
\newcommand{\kp}{k}
\newcommand{\pp}{\pi}
\newcommand{\zp}{\xi}

\newcommand{\ooo}{\mathcal O}

\date{}

\newcommand{\ce}{\mathbf{K}}
\newcommand{\et}{\lambda}
\newcommand{\emm}{w}

\newcommand{\dom}{\operatorname{dom}}
\newcommand{\cf}{\mathbf{F}}

\newcommand{\g}{G}

\newcommand{\Mmm}{\R^{\nn}_+}
\newcommand{\Mm}{\R^{\nn}_+}
\title[Duality]{A unified duality framework for barotropic,  quantum and Korteweg fluids}
\author[D.~Vorotnikov]{Dmitry Vorotnikov}
\address[D.~Vorotnikov]{University of Coimbra, CMUC, Department of Mathematics,  3001-501 Coimbra, Portugal}{}
\email{mitvorot@mat.uc.pt}

\begin{document}

\begin{abstract} We investigate a dual variational formulation, in the spirit of Brenier, for several compressible fluid models: the  compressible barotropic Euler system, the quantum Euler system, and the Euler–Korteweg system. We identify a unified abstract framework encompassing all three systems, which enables a simultaneous analysis. By introducing time-adaptive weights, we establish the consistency of the duality scheme on large time intervals. We prove the existence of variational dual solutions to the corresponding Cauchy problems for continuous, vacuum-free initial data in spaces of finite Radon measures, and establish the absence of a duality gap. As an application, we formulate and prove a “Dafermos principle” for these models: no subsolution can dissipate the total entropy earlier or at a faster rate than the corresponding strong solution on its interval of existence. We also discuss connections between our abstract consistency result and Brenier’s shock-free substitutes for entropy solutions of Burgers’ equation.
\end{abstract}
\maketitle


\textbf{MSC [2020] 35Q35, 37K58, 49Q99, 76N10}

\section{Introduction} 
Cauchy problems for nonlinear evolutionary PDEs often admit infinitely many weak solutions \cite{DP79,lel09,lel10,BNF16,DS17}. In the context of compressible fluid dynamics, such nonuniqueness was demonstrated in \cite{C14,CKL15,Fei1,DF15,CVC21,Fei25,KK20}, to name just a few. Many of these systems possess a formally conserved quantity, hereafter referred to as \emph{total entropy}, which may fail to be preserved for weak solutions unless some Onsager-type conditions are assumed. Although no universal criterion exists for selecting physically relevant solutions, various variational principles from natural and informational sciences can help to identify them. Prominent examples include Prigogine’s minimum entropy production principle for open systems \cite{Prigozhin} and Ziegler’s maximum entropy production principle for closed systems \cite{Ziegler},   together with numerous related principles \cite{MS06}. These principles have been employed both to derive PDE models and to select physically meaningful weak solutions \cite{glimm}.
Dafermos proposed to assume that physically relevant weak solutions dissipate total entropy more rapidly than irrelevant ones \cite{Daf1,Daf2}. While smooth solutions conserve the total entropy, weak solutions may dissipate it through shocks and other mechanisms. The criterion is local in time: at a bifurcation moment, the total entropies of ``good'' and ``bad'' solutions may coincide, after which the ``bad'' solution’s entropy temporarily exceeds that of the ``good'' solution, though this relationship may reverse at later times \cite{Hsiao,Daf2,Fei1,CJ16}. Although the criterion is known not to be fully compatible with the compressible barotropic Euler system \cite{CK14}, recent work has extended it to measure-valued solutions of this system \cite{FJ25}, and to subsolutions of some other models arising in fluid dynamics \cite{GK,CA23}; related numerical investigations have been reported in \cite{Klei22,Klei23,Klei24}. 
For systems of hyperbolic conservation laws, classical selection criteria (Lax’s shock condition, the Kruzhkov–Lax entropy criterion, and the vanishing viscosity criterion) are, under suitable assumptions, compatible with Dafermos principle \cite{Daf1,KS97}. In the context of Lagrangian and continuum mechanics, including the barotropic compressible Euler equations,  admissibility criteria based on minimization of Lagrangian action have recently been proposed or  examined in \cite{gi2024,mf25,gi2025,GHK}.

In this paper, we study the variational dual formulation of several compressible fluid models. The approach originates with Brenier \cite{CMP18, Br20}, who proposed, for the incompressible Euler system, the compressible barotropic Euler system, the Burgers equation, and related models, to consider solutions minimizing the time-integrated entropy. While this minimization may not always admit a solution, it naturally leads to a dual problem with more favorable convexity properties. His results imply the possibility of recovering smooth solutions of the compressible barotropic Euler system over small time intervals from the dual formulation, see Remark \ref{brenres} below. This method was applied to the multi-stream Euler–Poisson system in \cite{BMO22}, and a numerical implementation for the quadratic porous medium and Burgers equations was recently proposed in \cite{mirebeau_stampfli}. Further developments of Brenier’s approach, presented in \cite{V22, V25}, reveal structures in many nonlinear PDEs that allow variational dual formulations with advantageous properties. These dual problems are closely connected to optimal transport, specifically its \emph{ballistic} variant \cite{V22,V25, mirebeau_stampfli}.
A similar variational dual formulation has been proposed in \cite{Ach4,Ach6,Ach5} and related works. A central feature of the latter framework is the use of a \emph{base state}, which serves as an “initial guess” for the PDE solution; the entropy is then replaced by an appropriate relative entropy. In this view, the frameworks of \cite{CMP18, Br20, V22, V25, mirebeau_stampfli} use zero base states.  The recent preprints \cite{Ach7,AV25} build upon both approaches; among other contributions, \cite{AV25} introduces a duality-based gradient flow scheme for approximating solutions of PDEs for which the existence of global weak solutions may not yet have been established.
In \cite{V25}, we extended Brenier’s duality formalism to systems of PDEs that can be written in the abstract form \be \label{o:aeulerint}\p_t v= L(\cf(v)),\ee where the matrix-valued function $\cf$ satisfies suitable convexity and positivity conditions, and $L$ is a vector-valued differential operator of arbitrary order. We further assumed that \eqref{o:aeulerint} admits a strictly convex anisotropic \emph{entropy} functional $\ce(v)$ whose integral (i.e., the total entropy) is formally conserved along the flow. Moreover, the entropy $\ce$ was required to generate an anisotropic Orlicz space. Consequently, the analysis relied heavily on the theory of anisotropic Orlicz spaces (see, e.g., \cite{Gwiazda}). In particular, we established the existence of solutions to the dual problem belonging to an anisotropic Orlicz space. It was further shown that the NLS, NLKG, and GKdV equations with generic defocusing power-law nonlinearities fall within this framework.

However, the entropies arising in the context of compressible fluids typically take the form $$\ce(q,\rho)=\frac{|q|^2}{2\rho}+\dots,\quad \rho>0,$$ where $q$ represents the mass flux and $\rho$ is the density, cf. Section \ref{Sec3}, and therefore fail to generate an Orlicz space. The main motivation of the present work is to extend the formalism of \cite{V22, V25} to convex entropies of this type, thereby incorporating various compressible fluid models into the theory, well beyond Brenier’s original contributions.

The first model encompassed by our framework is the classical barotropic Euler system \eqref{cbf1}--\eqref{cbf3}, already mentioned above. To the best of our knowledge, the question of existence of global weak solutions for arbitrary initial data to this model remains open, cf. \cite{GW9,FJ25}, although for a dense set of initial data there are many ``wild'' weak solutions, cf. \cite{Ch2024,Fei25}. 

The second model under consideration is the quantum Euler system \eqref{nls1}--\eqref{nls3}, which is a classical model in quantum hydrodynamics. In this setting, one may impose the ansatz $v=\frac q \rho=\nabla \theta$, which characterizes irrotational flows. An important property of this system is that the above condition is formally preserved under the evolution. Moreover, in the irrotational case, the model is formally equivalent to the NLS equation via the M\"adelung transform \cite{AM18,Kh19,RenM}. Exploiting this correspondence, global existence of weak irrotational solutions was established in \cite{AM09,M2012}. However, even without the irrotationality assumption, one can prove the existence of weak solutions, including “wild” and non-unique ones, cf. \cite{DF15,BGL19,BT23}. In the present work, we aim to remain as general as possible and therefore do not assume that the initial data are irrotational.

A natural generalization of the quantum Euler system, which also fits within our framework, is the Euler–Korteweg system \eqref{kls1}–\eqref{kls3}. It describes capillary fluids and shares many of the same features as the quantum Euler system, cf. \cite{Saut}. It is known \cite{DF15} that, for any sufficiently smooth initial datum (allowing a vacuum set of Lebesgue measure zero), the Euler–Korteweg system admits infinitely many global weak solutions. An Onsager-type regularity condition ensuring that weak solutions of the Euler–Korteweg system conserve total entropy is proposed in \cite{Gwia17}.

In Section \ref{setti}, we introduce a general abstract framework that unifies all three systems, as well as numerous other models from \cite{CMP18,V22,V25}, enabling their simultaneous analysis. In Section \ref{duality}, by incorporating time-dependent weights, we establish the consistency of the duality scheme over large time intervals (Theorem \ref{t:smooth} and Remark \ref{r:timeei}). As an application, we formulate and prove a version of the “Dafermos principle” (Theorem \ref{dpr1}) for these models: no subsolution can dissipate the total entropy sooner or at a faster rate than a suitably defined \emph{strong solution} within the lifespan of the latter. In Section \ref{s:exi}, we demonstrate the solvability of the corresponding dual problems, for continuous vacuum-free initial data, in spaces of finite Radon measures, and show that no duality gap arises (Theorem \ref{t:exweak}). Section \ref{Sec3} provides a detailed explanation of how the compressible fluid systems mentioned above fit into our framework. In Section \ref{Burg}, we revisit Brenier’s treatment \cite{CMP18} of the duality associated with the inviscid Burgers equation and clarify certain points that may have remained ambiguous in his original work.  While Appendix \ref{apb} is purely technical, Appendix \ref{apa} concludes the paper by highlighting several challenging open questions.

\subsection*{Notation} Let us fix some basic notation. In the what follows, $\Omega$ is the periodic box $\mathbb T^d$, $d\in \mathbb N$, equipped with the Lebesgue measure $dx$. Let $d\mut:=dx\, dt$ be the Lebesgue measure on $[0,T]\times \Omega$. 
We will also use the notations $\R^{\nn}$, $\R^{\nn}_s$ and $\Mmm$ for the spaces of $N\times N$ matrices, symmetric matrices and positive semidefinite matrices, resp. For $A,B\in \R^{\nn}_s$, we write $A\geq B$ when $A-B\in \Mmm$ (the L\"owner order \cite{L30}), and $A:B$ for the scalar product of $A$ and $B$ generated by the Frobenius norm $|\cdot|$.  The symbol $I$ will stand for the identity matrix of a relevant size. The symbols $\ms(\mathcal Q),\ \ms(\mathcal Q;\R^n),\ \ms(\mathcal Q;\Mmm)$ will denote the spaces of finite signed, vector- or matrix-valued Radon measures, respectively, over a metric space $\mathcal Q$. Let $X:=C(\Omega)$ (the space of continuous functions). Note that $X^*=\ms(\Omega)$.  For two measurable functions $u,v:\Omega\to \R^m$, $m\in \mathbb N$, we use the shorthand $$(u,v):=\int_\Omega u(x)\cdot v(x)\,dx.$$ In a similar spirit, for $\mu \in X^*$ and $f\in X$, we denote $$(f,\mu):=\int_\Omega f\,d\mu.$$ To avoid confusion with space-time duality, we reserve the notation $$\la f,\mu \ra :=\int_{[0,T]\times \Omega} f\,d\mu$$ for $v\in C([0,T]\times \Omega;\R^m)$ and $\mu \in \ms([0,T]\times \Omega;\R^m)$.
\section{Basic framework}

\label{setti}

Before presenting the duality scheme, we introduce in this section a convenient abstract framework that encompasses the compressible fluid models under consideration.

\begin{remark} \label{contsub} For purely technical reasons, in the ``primal'' setting  we restrict attention to continuous solutions (and, later, subsolutions). However, the dual formulation will be constructed in such a way that discontinuous solutions can also be recovered. Additional details  on this sensitive issue can be found in Section \ref{s:bur} and Appendix \ref{apa}. \end{remark}

We begin with a preliminary, semi-rigorous formulation; the precise definitions will be given below.

\begin{setting} Let
	\begin{enumerate} \item $n, N\in \mathbb N,\ Z\in \mathbb Z_{\ge 0},$ \item  $\cf:\dom \cf \subset \R^n\to \Mm$ be a L\"owner-convex \cite{KR36} matrix function,  \item $\dom \cf$ be convex and have a non-empty interior $\ooo$, \item  $\cf$ be continuously  differentiable  in $\ooo$,  \item $L: D(L)\subset X_s^{\nn}\to X^n$ be a closed densely defined linear operator, \item $\lc: D(\lc)\subset X^n \to X^Z$ be a closed densely defined linear operator with closed range, \item the initial condition satisfy
		\be \label{incond} v_0\in  C(\Omega;\dom \cf)\cap \ker \lc. \ee \end{enumerate} We seek a continuous function $$v:[0,T]\to X^n$$ solving the initial-value problem \be \label{e:aeuler}\p_t v= L(\cf(v)), \quad v(0,\cdot)=v_0\in X^n,\ee subject to the constraints \begin{gather} \label{e:aeulerlc} \lc v=0,\\ \label{e:aeulero}  v(t,x)\in \dom \cf.\end{gather} 
	\end{setting}

\begin{assumption} For simplicity, we will tacitly assume that $L$ and $\lc$, together with their adjoints $L^*$ and $\lc^*$, are differential or integral operators (with respect to the spatial variable) defined on their natural domains. In particular, they can be lifted to act on time-dependent functions. \end{assumption}


\begin{conv} Define the \emph{entropy} function by \be \label{e:entr1}\ce: \R^n \to \R_+,\ \ce(v):=\begin{cases}\frac 1 2 \tr \cf(v), & v\in \dom \cf, \\ +\infty, & v\not\in \dom \cf. \end{cases}\ee It follows from the L\"owner convexity assumption that $\ce$ is convex. \end{conv}

		\begin{remark} {Here we adopt a slightly cleaner definition of entropy than in \cite{V25}. There,  the technical assumption $\ce(0)=0$ that was required to work in Orlicz spaces, which are not employed in this paper,  forced us to define the entropy as
				$
				\ce=\frac 1 2 (\tr (\cf(\cdot)-\cf(0))).
				$ 
		} Of course, this is just a minor cosmetic issue.  \end{remark}
\begin{assumption} \label{asco} We assume that $\ce$ is strictly convex (but not necessarily uniformly convex) in $\ooo$ and lower-semicontinuous in $\R^n$. \end{assumption}
 \begin{conv} We now define a transformation\footnote{See Remark \ref{remsharp}.} that will serve as a systematic change of variables. To any given $v\in \ooo$ we associate the vector \be \mom:=\nabla \ce(v), \label{sharp}\ee where the gradient $\nabla$ is taken with respect to $v$. 
 Since $\ce$  is strictly convex in $\ooo$, the map $$\nabla \ce: v \mapsto \mom$$ is $C^1$-smooth and injective. Moreover, $$v=\nabla \ce^*(\mom)$$ for any $\mom\in \nabla \ce(\ooo)$, where $\ce^*$ is the Legendre transform of $\ce$. \end{conv}
 
 \begin{remark} Note that $\nabla \ce(\ooo)$ need not be convex, since full Legendre duality \cite[Section 26]{rock} is not imposed, being incompatible with the applications in Section~\ref{Sec3}. \end{remark}
 
 {\begin{remark} \label{remsharp} The transformation \eqref{sharp} is well known in the theory of systems of conservation laws. It originates in the work of Godunov \cite{God1}, see also \cite{Godtrib} for a modern survey. The resulting variables were originally introduced to rewrite systems of conservation laws in symmetric form and are known by several names, including \emph{entropy variables} \cite{Tad87,Tad26}, \emph{main field} \cite{RS81}, and \emph{conjugate variables} \cite{Gav24}.
 		In this paper, we refer to them as \emph{sharp} variables, by loose analogy with the musical isomorphism in differential geometry. We stress that, in this paper, we employ the change of variables \eqref{sharp} well beyond the classical realm of first-order systems of conservation laws, and that our strategy is entirely unrelated to symmetrization.  \end{remark}}

 \begin{conv}  Define the \emph{total entropy} by \begin{equation} \label{e:consp1} K(t):=\int_\Omega \ce (v(t,x))\,d x.\end{equation}
  \end{conv}
  
 
\begin{assumption}  \label{a:consc} We will focus on the situation when $L$ satisfies the formal \emph{conservativity} condition \be \label{e:acons} \int_\Omega \cf(v): (L^*(\mom))\,dx=0\ee for all  \emph{smooth} vector fields $v: \Omega \to \ooo$  satisfying the linear constraint \eqref{e:aeulerlc}. \end{assumption}
Assumption \ref{a:consc} and system \eqref{e:aeuler}--\eqref{e:aeulero} formally imply the following relation for the sharp variable $\mom$:
\be \label{e:aeuler2}\p_t (\mom\mut)_l+L^* (\mom\mut):\partial_l \cf(\nabla \ce^*(\mom))=(\lc^* \pp)_l, \quad l=1,\dots, n.\ee Here $\pp:[0,T]\to (X^Z)^*$ plays the role of a Lagrange multiplier associated with the linear constraint \eqref{e:aeulerlc}; see Remark \ref{sharpq} for an explicit example. Moreover, \eqref{e:acons} formally implies that the total entropy is conserved: \be \label{e:entc} K(t)=\int_\Omega \ce (v_0)\,d x=:K_0,\ t\in[0,T].\ee We present the  formal proofs of \eqref{e:aeuler2} and \eqref{e:entc} in Appendix \ref{apb}.

Problem \eqref{e:aeuler}--\eqref{e:aeulero} admits the following natural weak formulation: \be \label{e:w1}\int_0^T \left[(v,\lc^* w)+(v,\p_t a)+(\cf(v), L^*a)\right]\, dt+(v_0, a(0))=0\ee for all smooth vector fields $a: [0,T]\to X^n$, $a(T)= 0$, and $w: [0,T]\to X^Z$.

Let us now rewrite problem \eqref{e:w1} in terms of the test functions $B:=L^*a$ and $E:=\p_t a+\lc^* w$. The relationship between $B$ and $E$ can alternatively be characterized by the conditions \be\label{e:constr1}\p_t B-L^* E\in \operatorname{ran} (L^*\lc^*), \quad B(T)=0.\ee  Now observe that \begin{multline} \label{e:vab} (v_0, a(0))=-\int_0^T \left(v_0,   \p_t a\right)\,dt\\ =-\int_0^T \left(v_0,   \p_t a\right)\,dt-\int_0^T \left(\lc v_0,   w\right)\,dt=-\int_0^T \left(v_0,   E\right)\,dt. \end{multline}  
Hence, \eqref{e:w1} becomes \be \label{e:w2sm}\int_0^T \left[(v-v_0,E)+(\cf(v), B)\right]\,dt=0\ee for all smooth vector fields $B: [0,T]\to X^{\nn}_s$, $E: [0,T]\to X^n$ satisfying the constraints \eqref{e:constr1}.

Note that \eqref{e:constr1} admits the following weak formulation, which remains meaningful when $B$ and $E$ are merely finite Radon measures on $[0,T]\times \Omega$ taking values in the corresponding Euclidean spaces: \be \label{e:constrweak}\la \p_t \Psi, B \ra+\la L \Psi, E \ra=0, \ee for all smooth vector fields $\Psi: [0,T]\to X^{\nn}_s$, $\lc L \Psi=0$, $\Psi(0)= 0$. Consequently, \eqref{e:w2sm} also admits a  weak form: \be \label{e:w2}\la v-v_0,E \ra+\la \cf(v), B \ra=0\ee for all pairs $ (E,B)$ subject to constraint \eqref{e:constrweak}.

Motivated by this discussion, we adopt the following definitions, where we tacitly assume  \eqref{incond}.
\begin{definition}[Weak solutions] A function \be \label{e:vclass} v\in C([0,T]\times \Omega;\R^n) \ee fulfilling \eqref{e:aeulero} is a \emph{weak solution} to \eqref{e:aeuler}--\eqref{e:aeulero} if \be \label{freg} \cf(v)\in  C([0,T]\times \Omega;\Mm)\ee and \eqref{e:w2} holds for all pairs \be \label{e:be} (E,B)\in \ms ([0,T]\times \Omega;\R^n)\times \ms([0,T]\times \Omega;\R^{\nn}_s)\ee subject to constraint \eqref{e:constrweak}. \end{definition}

\begin{definition}[Subsolutions] A pair of functions $$(v,M)\in C([0,T]\times \Omega;\R^n)\times C([0,T]\times \Omega;\Mm),$$ satisfying  $$\cf(v)\leq M,$$ is a \emph{subsolution} to \eqref{e:aeuler}--\eqref{e:aeulero} provided \be \label{e:w2sub}\la v-v_0,E \ra+\la M, B \ra=0\ee holds for all pairs \be \label{e:be3} (E,B)\in \ms ([0,T]\times \Omega;\R^n)\times \ms([0,T]\times \Omega;\R^{\nn}_s)\ee subject to constraint \eqref{e:constrweak}. \end{definition}
Obviously, if $v$ is a weak solution, then $(v,\cf(v))$ is a subsolution. Accordingly, to be consistent with the preceding discussion, we define the total entropy of a subsolution by\begin{equation} \label{tess}  \tilde K(t):=\frac 1 2 \int_\Omega \tr (M(t))\,dx.\end{equation}

\begin{conv}[Weight] Fix a smooth function $\wei:[0,T]\to \R_+$ that is bounded away from $0$ and $\infty$.  Let $\Wei(t):=\int_t^T \wei(s)\, ds$. \end{conv}

One does not expect all weak solutions to satisfy the important relations \eqref{e:aeuler2} and \eqref{e:entc}.
We reserve the term \emph{strong solution} for those that do, in the following sense.

\begin{definition}[Strong solutions] \label{d:strongclass} Assume that $v_0\in C(\Omega;\ooo)\cap \ker \lc$. A function \be v\in C([0,T]\times \Omega; \ooo) \label{e:vclassss}\ee satisfying \be\label{e:strongclass} \p_t (\Wei\mom\mut)\in \ms ([0,T]\times \Omega;\R^n), \Wei L^* (\mom\mut)\in \ms ([0,T]\times \Omega; \R^{\nn}_s) \ee and \be\label{e:amvm} c I \mut +\Wei  L^*(\mom\mut)\geq 0\ \mathrm{as}\ \mathrm{symmetric-matrix-valued}\ \mathrm{measures}\ee with a uniform constant $c\geq 0$ is called a \emph{strong solution} of  \eqref{e:aeuler}--\eqref{e:aeulero} if \begin{enumerate} \item $v$ it is a weak solution, \item  \eqref{e:entc} holds, \item \eqref{e:aeuler2} is fulfilled with some $\pp$ such that \be\label{e:strongclasspi} \Wei\pp \in \ms ([0,T]\times \Omega;\R^Z), \ \Wei\lc^*\pp \in \ms ([0,T]\times \Omega;\R^n).\ee \end{enumerate} \end{definition}

\begin{remark} At one point, we will slightly abuse terminology by referring to the corresponding pair $(v,\pi)$ as a strong solution.\end{remark}

\begin{remark} It is obvious that a smooth  solution to  \eqref{e:aeuler}--\eqref{e:aeulero} with values in $\ooo$ is a strong solution in the sense of Definition \ref{d:strongclass}. We aim to work with a notion of strong solution that is as weak as possible while still ensuring the properties needed below. To fix the ideas, we restrict this definition to functions taking values in~$\ooo$. In contrast, our (continuous) weak solutions are allowed to take values outside $\ooo$. We expect that the stronger constraint imposed on strong solutions is not essential and can be removed or relaxed with additional technical effort, including appropriate modifications of the definitions. \end{remark}

\begin{remark} \label{r:r25} Since $\Wei(T)=0$,  requirement \eqref{e:strongclass} should be understood in the sense that there exist a vector-valued measure, denoted $\p_t (\Wei\mom\mut)$, and a symmetric-matrix-valued measure, denoted $\Wei L^* (\mom\mut)$, such that the identities $$\la \phi, \p_t (\Wei\mom\mut) \ra=-\la   \p_t \phi , \Wei\mom \mut \ra,$$ 
	$$\la \Phi, \Wei L^* (\mom\mut)\ra=\la   L \Phi , \Wei \mom\mut\ra$$ hold for any pair $(\phi,\Phi)\in C^\infty([0,T]\times \Omega;\R^n\times \R^{\nn}_s)$ with $\phi(0)=0$.  \end{remark} \begin{remark} \label{r:regul1} Let us explain why \eqref{e:aeuler2} makes sense for $v$ and $\pp$ satisfying \eqref{e:vclassss}, \eqref{e:strongclass}, \eqref{e:strongclasspi}. Obviously, \eqref{e:vclassss} implies $$\mom \in C([0,T]\times \Omega;\R^n)$$ and $\cf(\nabla \ce^*(\mom))= \cf(v)\in C([0,T]\times \Omega; \R^{\nn}_s)$. Thus  \eqref{e:aeuler2} is rigorously meaningful when written in the form \be \label{o:aeuler2w}\wei (\mom\mut)_l+\p_t (\Wei\mom\mut)_l+\Wei L^* (\mom\mut):\partial_l \cf(\nabla \ce^*(\mom))=(\Wei\lc^* \pp)_l,\ee interpreted as equality of  finite Radon measures.  \end{remark}
	
	\begin{remark} \label{r:equiv} Definition \ref{d:strongclass} is actually independent of the weight function $\Wei$. Indeed, let $v$ be a strong solution, $\wei_1$ be another positive scalar function bounded away from $0$ and $\infty$, and $\Wei_1=\int_t^T \wei_1(s)\, ds$. Obviously, \be \Wei(t) \lesssim  \Wei_1(t)\lesssim \Wei(t). \label{weights}\ee Hence, $$\Wei_1 L^* (\mom)\in \ms ([0,T]\times \Omega; \R^{\nn}_s),$$ and  \begin{multline*} \p_t (\Wei_1\mom)=-\wei_1 \mom\mut+ \Wei_1 \p_t (\mom)\\=-\wei_1 \mom\mut+\frac {\Wei_1} {\Wei}\wei \mom\mut+\frac {\Wei_1} {\Wei} \p_t (\Wei\mom)\in \ms ([0,T]\times \Omega;\R^n).\end{multline*} Moreover, \eqref{weights} implies \eqref{e:amvm} for $\Wei_1$ (possibly with a different uniform constant $c$) and also yields \eqref{e:strongclasspi} for $\Wei_1$. \end{remark}

\section{Duality} \label{duality}

In order to set up the duality framework, we first need to specify the primal problem. It consists in finding a weak solution to \eqref{e:aeuler}–\eqref{e:aeulero}, regardless of the fact that there might be no global existence for arbitrary initial data, that minimizes a suitable weighted time integral of the total entropy, namely $\int_0^T \wei(t) K(t)\, dt$. (As noted in \cite[Remark 3.12]{V25}, this idea, for appropriate $\wei$, can be viewed as a ``rough'' Dafermos principle). 
The task of selecting a weak solution that minimizes $\int_0^T \wei(t) K(t)\, dt$ can be implemented via the saddle-point problem \be\label{e:sadd1}\mathcal I(v_0, T)=\inf_{v}\sup_{E,B:\,\eqref{e:constrweak}}\left[\la v-v_0,E \ra +\frac 1 2 \la \cf(v), \wei I\mut+2B\ra\right].\ee The infimum in \eqref{e:sadd1} is taken over all $v\in C([0,T]\times \Omega;\R^n)$ fulfilling \eqref{e:aeulero} and \eqref{freg}, and the supremum is taken over all pairs $(E,B)$ satisfying \eqref{e:be} and the linear constraint \eqref{e:constrweak}. 
(If a weak solution does not exist, then $\mathcal I(v_0,  T)=+\infty$).

The dual problem is \be\label{e:sadd2}\mathcal J(v_0, T)=\sup_{E,B:\,\eqref{e:constrweak}}\inf_{v}\left[\la v-v_0,E \ra +\frac 1 2 \la \cf(v), \wei I\mut+2B\ra\right],\ee where $v,E,B$ are varying in the same function spaces as above. 

We also consider a “relaxed” version of the primal problem, consisting in selecting subsolutions that minimize the weighted time integral of the total entropy $\int_0^T \wei(t) \tilde K(t)\, dt$. This 
 can be realized through the saddle-point problem \be\label{e:sadd1sub}\tilde{\mathcal I}(v_0, T)=\inf_{v,M:\cf(v)\leq M}\sup_{E,B:\,\eqref{e:constrweak}}\left[\la v-v_0,E \ra +\frac 1 2 \la M, \wei I\mut+2B\ra\right].\ee The infimum in \eqref{e:sadd1sub} is taken over all $v\in C([0,T]\times \Omega;\R^n)$ and $M\in C([0,T]\times \Omega;\R^{\nn}_s)$, and the supremum is taken over all pairs $(E,B)$ satisfying \eqref{e:be} and  \eqref{e:constrweak}.

The dual problem is \be\label{e:sadd2sub}\tilde{\mathcal J}(v_0, T)=\sup_{E,B:\,\eqref{e:constrweak}}\inf_{v,M:\cf(v)\leq M}\left[\la v-v_0,E \ra +\frac 1 2 \la M, \wei I\mut+2B\ra\right],\ee where $v,M,E,B$ are varying in the same function spaces as above. 

\begin{remark} Since $\inf\sup\geq \sup\inf$, one has \be \label{ijj} \mathcal I(v_0, T)\ge \tilde{\mathcal I}(v_0, T)\ge \tilde{\mathcal J}(v_0, T), \ \mathcal I(v_0, T)\ge \mathcal J(v_0, T)\ge \tilde{\mathcal J}(v_0, T).\ee Note that the $\sup$ in \eqref{e:sadd1} is always $+\infty$ if $v$ is not a weak solution, whence $\mathcal I(v_0,T)=+\infty$ if there are no weak solutions. The same applies to \eqref{e:sadd1sub} and $\tilde{\mathcal I}(v_0, T)$ in the case of  subsolutions. On the other hand, if  there exists a strong solution $v$, then the corresponding $\sup$ in \eqref{e:sadd1} is equal to $\frac 12\la \cf(v), \wei I \mut \ra=\int_0^T \wei K(t)\,dt=\int_0^T \wei K_0\,dt=\Wei(0)K_0$, which yields $\mathcal I(v_0,T)\leq \Wei(0)K_0$. Finally, testing by $(E,B)=(0,0)$ we see that $$\tilde{\mathcal J}(v_0, T)\geq \frac 1 2 \inf_{v,M:\cf(v)\leq M}\la M, \wei I\mut \ra \geq \inf_{v,M:\cf(v)\leq M}\la \cf(v), \wei I\mut\ra \geq 0$$ because $\cf(v(t,x))$ is positive semidefinite.    \label{rweak}\end{remark} 

It is easy to see that any solution to \eqref{e:sadd2sub} necessarily satisfies \be\label{e:bwe} \wei I\mut+2B\geq 0\ \mathrm{as}\ \mathrm{symmetric-matrix-valued}\ \mathrm{measures}.\ee
Consider the nonlinear functional $$\mathcal K: \ms([0,T]\times \Omega;\R^n)\times \ms([0,T]\times \Omega;\R^{\nn}_s)\to \R$$ defined by the formula
\be\label{e:defk1} \mathcal K(E,B)=\inf_{\cf(z)\le M} \la z,E \ra +\frac 1 2 \la M, \wei I\mut+2B\ra,\ee  where the infimum is taken over all pairs $(z,M)\in C([0,T]\times \Omega;\R^n)\times C([0,T]\times \Omega;\R^{\nn}_s)$.

Then \eqref{e:sadd2sub} is equivalent to \be\label{e:conc}\tilde{\mathcal J}(v_0, T)=\sup_{E,B:\,\eqref{e:constrweak},\eqref{e:bwe}}-\la v_0,E \ra +\bigk,\ee the supremum is taken over all pairs $(E,B)$ belonging to the class \eqref{e:be}.

The following theorem shows that a strong solution to problem \eqref{e:aeuler}--\eqref{e:aeulero} on a certain, possibly small, time interval $[0,T]$  determines a solution to the dual optimization problem \eqref{e:conc} in a neat, explicit way. Furthermore, the \textbf{interval need not be small} provided that the weight is suitably chosen, see Remark \ref{r:timeei}. 

\begin{theorem}[Consistency] \label{t:smooth} Assume \eqref{e:acons}.  Let $v_0\in C(\Omega;\ooo)\cap \ker \lc$. Let $(v,\pi)$ be a strong solution to \eqref{e:aeuler}--\eqref{e:aeulero}   satisfying \be\label{e:pd++} \wei I \mut +2\Wei (t) L^*(\mom\mut)\ge 0\ \ \mathrm{as}\ \mathrm{symmetric-matrix-valued}\ \mathrm{measures}.\ee  Then $\mathcal I(v_0,T)=\tilde{\mathcal J}(v_0,T)=\tilde{\mathcal I}(v_0,T)=\mathcal J(v_0,T)=\Wei(0)K_0$. The pair $(E_+,B_+)$ defined by $$B_+=\Wei  L^*(\mom\mut),\, E_+=\p_t (\Wei \mom\mut)-\Wei \lc^* \pp,$$  is a maximizer of \eqref{e:conc}.  Moreover, these formulas can be inverted to represent\footnote{See Remark \ref{remdu}.} $\mom$ in terms of $E_+$ on $[0,T)\times \Omega$ via the duality relation
	\be \label{e:ta1} \la\psi, \mom\mut \ra= -\la\int_0^t\frac {\psi(s,\cdot)}{\Wei(s)}\,ds, E_+\ra,\ee for any smooth test function $\psi(t,x)$ that vanishes in a neighborhood of $\{T\}\times \Omega$ and satisfies $\lc \psi=0$.

\end{theorem}

\begin{proof} 
 By construction,  the pair $(E_+,B_+)$ belongs to $\ms([0,T]\times \Omega;\R^n)\times \ms([0,T]\times \Omega;\R^{\nn}_s)$.  Let us show that this pair verifies  \eqref{e:constrweak}. Indeed, by Remark \ref{r:r25},  \begin{multline*} \la \p_t \Psi, B_+ \ra+\la L \Psi, E_+ \ra=  \la \p_t \Psi, \Wei  L^*(\mom\mut) \ra+\la L \Psi,\p_t (\Wei \mom\mut)-\Wei \lc^* \pp \ra\\\ = -\la \lc L \Psi,\Wei ^* \pp \ra =0\end{multline*} for all smooth vector fields $\Psi: [0,T]\to X^{\nn}_s$, $\lc L \Psi=0$, $\Psi(0)= 0$.
 
 Moreover, \eqref{e:pd++} implies \eqref{e:bwe} for $B_+$. We now claim that \be \label{e:vret1} \frac 1 2 (\wei I\mut+2B_+):\partial_l\cf(v)+(E_+)_l=0,\quad l=1,\dots,n,\ee as Radon measures.  Indeed, using \eqref{o:aeuler2w} we compute \begin{multline}\frac 1 2 (\wei I\mut+2B_+):\partial_l\cf(v)+(E_+)_l\\=\frac 1 2 (\wei I\mut+2\Wei L^*(\mom\mut)):\partial_l\cf(v)+\p_t (\Wei \mom\mut)_l-(\Wei \lc^* \pp)_l\\ =\wei \partial_l\ce(v)\mut +\Wei L^*(\mom\mut):\partial_l\cf(v)+\p_t (\Wei \mom\mut)_l-(\Wei \lc^* \pp)_l \\= \wei (\mom\mut)_l+\p_t (\Wei\mom\mut)_l+\Wei L^* (\mom\mut):\partial_l \cf(\nabla \ce^*(\mom))-(\Wei\lc^* \pp)_l
 	=0.\end{multline}
 
On the other hand, $v$ satisfies \eqref{e:w2} with test functions $(E_+,B_+)$ because it is, in particular, a weak solution. Thus we have \be \label{e:w2+}\la v-v_0,E_+ \ra+\la \cf(v), B_+\ra=0.\ee Hence, by \eqref{e:vret1}, \be \label{f45} -\la v_0,E_+\ra +\la \cf(v), B_+\ra =\frac 1 2 \sum_{l=1}^n \la v_l\partial_l\cf(v),\wei I\mut+2B_+\ra.\ee Employing \eqref{f45} and \eqref{e:entc}, we obtain   \begin{multline} \label{e:w3+}\frac 1 2 \la \cf(v),\wei I\mut+2B_+\ra -\frac 1 2 \sum_{l=1}^n \la v_l\partial_l\cf(v),\wei I\mut+2B_+\ra  \\= \la v_0,E_+ \ra+\la \wei(t)\ce (v(t)), \mut\ra =\la v_0,E_+\ra + \int_0^T\wei(t)K(t)\,dt=\la v_0,E_+\ra + \Wei(0)K_0.\end{multline} 

By Remark \ref{rweak}, we have $\mathcal I(v_0, T)\leq \Wei(0)K_0$. Thus, in view of \eqref{ijj}, it suffices to prove that \be \label{e:claim1}  -\la v_0,E_+\ra +\bigkp\ge \Wei(0)K_0, \ee which ensures that there is no gap between $\mathcal I$ and $\tilde{\mathcal J}$. Indeed, the L\"owner convexity of $\cf$ means that the function $$\Phi: y\to \cf(y):P$$ is convex with respect to $y\in \dom \cf$ for any matrix $P\ge 0$. Consequently, for any parameter $\xi\in \ooo$,  the function $$y\mapsto -y\cdot \nabla \Phi (\xi)+\Phi (y)$$ attains its minimum at $y=\xi$. 
Leveraging this fact and using \eqref{e:vret1}, \eqref{e:w3+} and \eqref{e:bwe} for $B_+$, we conclude \begin{multline*} -\la v_0,E_+\ra +\bigkp\\=-\la v_0,E_+\ra +\inf_{\cf(z)\le M}\left[-\frac 1 2\sum_{l=1}^n\la z_l \partial_l\cf(v),\wei I\mut+2B_+\ra+\frac 1 2 \la M, \wei I\mut+2B_+\ra\right]\\ \ge  -\la v_0,E_+\ra +\inf_{z\in C([0,T]\times \Omega;\dom \cf)} \left[-\frac 1 2\sum_{l=1}^n\la z_l \partial_l\cf(v),\wei I\mut+2B_+\ra+\frac 1 2 \la \cf(z), \wei I\mut+2B_+\ra\right]\\
\\\ge -\la v_0,E_+\ra -\frac 1 2\sum_{l=1}^n\la v_l \partial_l\cf(v),\wei I\mut +2B_+\ra+\frac 1 2 \la \cf(v), \wei I\mut +2B_+\ra \\=\Wei(0)K_0. \end{multline*} 

Finally, using Remark \ref{r:r25} we get \begin{multline*}\la\psi, \mom\mut \ra=\la\frac \psi \Wei, \Wei\mom\mut \ra=\la\p_t\left(\int_0^t\frac { \psi(s,\cdot)}{\Wei(s)}\,ds\right), \Wei\mom\mut \ra+\la \lc\left(\int_0^t\frac { \psi(s,\cdot)}{\Wei(s)}\,ds\right), \Wei\pp\ra \\ = -\la\int_0^t\frac {\psi(s,\cdot)}{\Wei(s)}\,ds, \p_t (\Wei \mom\mut)\ra + \la\int_0^t\frac {\psi(s,\cdot)}{\Wei(s)}\,ds, \Wei \lc^*\pp\ra = -\la\int_0^t\frac {\psi(s,\cdot)}{\Wei(s)}\,ds, E_+\ra,\end{multline*}  for any smooth function $\psi$ that vanishes in a neighborhood of $\{T\}\times \Omega$ and satisfies $\lc \psi=0$. \end{proof}

\begin{remark}[Adapting the weight] \label{r:timeei}   Let $v$ be a strong solution to \eqref{e:aeuler}--\eqref{e:aeulero}.  Then the weight $\wei$ can be selected to guarantee the consistency on any interval $[0,T_1]$, $T_1<T$, cf. a similar feature in \cite{V25}. Indeed, by Remark \ref{r:equiv} with $\wei_1
	\equiv 1$,  for any strong solution on $[0,T]$ we have that \be\label{e:amvmcc} c I \mut +(T-t)  L^*(\mom\mut)\geq 0\ \mathrm{as}\ \mathrm{symmetric-matrix-valued}\ \mathrm{measures}\ee with a uniform constant $c\geq 0$. Thus, for any $T_1<T$, there exists $\gamma>0$ such that \be \label{e:weiwei} \gamma I\mut +2 L^*(\mom\mut)\ge 0\ee on $[0,T_1]\times \Omega$. It is straightforward from Definition \ref{d:strongclass} that $v$ (more accurately, the restriction $v|_{[0,T_1]\times \Omega}$) is a strong solution on  $[0,T_1]$.  Consider the weight \be \wei(t):=\exp(-\gamma t)\ee and let $\Wei(t):=\int_t^{T_1} \wei(s)\, ds\geq 0$ (note that we are now working on the interval $[0,T_1]$). It is easy to see that \be \label{1l}  \gamma\Wei\leq \wei.\ee From \eqref{e:weiwei} and \eqref{1l}, we get \eqref{e:pd++} that yields the consistency on $[0,T_1]$.  \end{remark}

\begin{remark} \label{remdu} Since the test functions $\psi$ should satisfy the constraint $\lc \psi=0$, $\mom$ can only be determined from \eqref{e:ta1} up to an additive term lying in the range of $\lc^*$. However, to retrieve the correct solution one can use the information that $\lc v=0$. See also Open Problem \ref{r:time}. This issue is absent when there is no linear constraint, i.e., when $\lc\equiv 0$. \end{remark}

\begin{theorem}[Dafermos principle] \label{dpr1} Let $v_0\in C(\Omega;\ooo)\cap \ker \lc$. Let $v$ be a strong solution to \eqref{e:aeuler}--\eqref{e:aeulero} with total entropy \eqref{e:consp1} and $(u,M)$ a subsolution with total entropy \eqref{tess}. Then, for any $0\leq t_0<t_1\leq T$, it cannot simultaneously be that $\tilde K(t)\le K(t)$ for $t\in [0,t_0]$ and $\tilde K(t)<K(t)$ for $t\in(t_0,t_1)$. In particular, it is impossible that $\tilde K(t)<K(t)$ for $t\in(0,\epsilon)$, $\epsilon>0$. \end{theorem}

	\begin{proof} The proof is similar to that of \cite[Theorem 4.3]{V25}, but we include it for completeness. We prove the first claim; the second is obtained by setting $t_0=0$ and $t_1=\epsilon$. 
		
		Assume that the  subsolution $(u,M)$ satisfies $\tilde K(t)\le K(t)$ for $t\in [0,t_0]$ and $\tilde K(t)<K(t)$ for $t\in(t_0,t_1)$ for some $0\leq t_0<t_1\leq T$. W.l.o.g. $t_1<T$. Fix any $T_1\in (t_1,T)$. Let $\gamma$ be a sufficiently large number that satisfies \eqref{e:weiwei} on $[0,T_1]\times \Omega$ and other lower bounds to be determined below. As in Remark \ref{r:timeei}, we set $\wei(t):=\exp(-\gamma t)$ and $\Wei(t):=\int_t^{T_1} \wei(s)\, ds$, and infer that \eqref{e:pd++} holds on $[0,T_1]\times \Omega$.   By Theorem \ref{t:smooth}, $\tilde{\mathcal I}(v_0,T_1)=\Wei(0)K_0$. Thus in order to get a contradiction we just need to prove that \be \label{dgoal} \int_0^{T_1} \wei(t) \tilde K(t)\, dt< \Wei(0)K_0.\ee
		
		Since $v$ is a strong solution, its total entropy $K(t)\equiv K_0$, so \eqref{dgoal} is equivalent to $$ \int_0^{T_1} \wei(t) (\tilde K(t)-K(t))\, dt< 0.$$ Since $\tilde K(t)\le K(t)$ for $t\in [0,t_0]$, it is enough to prove \be \label{dgoal2} \int_{t_0}^{T_1} \wei(t) (\tilde K(t)-K(t))\, dt< 0.\ee 
	Let $$\epsilon:=\int_{t_1}^{T_1} \wei(t) (\tilde K(t)-K(t))\, dt.$$ 
	Then our claim \eqref{dgoal2} becomes \be \label{dgoal3} \int_{t_0}^{t_1} \wei(t) (\tilde K(t)-K(t))\, dt< -\epsilon.\ee 
If $\epsilon\leq 0$, \eqref{dgoal3} is obvious, so let us assume $\epsilon>0$. Observe now that $$\epsilon\leq \wei(t_1) \int_{0}^T  |\tilde K(t)-K(t)|\, dt\leq C_\epsilon \wei(t_1),$$ where $C_\epsilon$ does not depend on $\gamma$. Fix any point $t_2\in (t_0,t_1)$. Let $\gamma$ be so large that $$\exp(\gamma(t_1-t_2))\int_{t_0}^{t_2} ( K(t)-\tilde K(t))\, dt\ge C_\epsilon.$$  Then we conclude that \begin{multline*} \int_{t_0}^{t_1} \wei(t) (\tilde K(t)-K(t))\, dt<  \int_{t_0}^{t_2} \wei(t) (\tilde K(t)-K(t))\, dt\\ \leq \wei(t_2) \int_{t_0}^{t_2}  (\tilde K(t)-K(t))\, dt=\exp(-\gamma t_2) \int_{t_0}^{t_2}  (\tilde K(t)-K(t))\, dt \\ \leq -\exp(-\gamma t_1)C_\epsilon\leq -\epsilon.\end{multline*}
		\end{proof}

\section{Existence and absence of a duality gap} \label {s:exi}

In this section, we establish\footnote{See, however, Open Problem \ref{op1}.} the existence of solutions to the dual problem, which may be called \emph{variational dual} solutions to \eqref{e:aeuler}–\eqref{e:aeulero} (cf. \cite{Ach5,Ach7,AV25}), and demonstrate that no duality gap arises between $\tilde{\mathcal I}(v_0, T)$ and $\tilde{\mathcal J}(v_0, T)$. Note that here we obtain only a restricted form of absence of a duality gap: unlike Theorem \ref{t:smooth} and Remark \ref{nodualitygap}, we cannot claim equality between $\mathcal I(v_0, T)$ and $\tilde{\mathcal I}(v_0, T)$.

\begin{theorem}[Existence and absence of a duality gap] \label{t:exweak}  Assume that \be \label{e:plo} L(I)=0.\ee 
	Then for any $v_0\in  C(\Omega;\ooo)\cap\ker \lc$ there exists a maximizer $$(E,B)\in  \mathcal M([0,T]\times \Omega;\R^n)\times \mathcal M([0,T]\times \Omega;\R^{\nn}_s)$$ of \eqref{e:conc}, and there is no duality gap in the sense that \be \label{ndg} 0\leq \tilde{\mathcal I}(v_0, T)=\tilde{\mathcal J}(v_0, T)<+\infty.\ee \end{theorem}
\begin{proof} 
	Let $\mathcal Z:=C([0,T]\times \Omega;\R^n)\times C ([0,T]\times \Omega;\R^{\nn}_s)$. Consider the following two \emph{convex} functionals on ${\mathcal Z}$:
$$\Theta(z,M)= \left\{\begin{array}{l} 
	\frac 1 2 \la M,\wei I\mut \ra \quad \mathrm{if}\ M\ge \cf(z)\ \mathrm{pointwise},\\
	+\infty\quad \mathrm{otherwise},
\end{array} \right.$$
$$\Xi(z,M)=\sup_{E,B:\,\eqref{e:constrweak}}\la z-v_0,E \ra +\la M, B \ra.$$ Indeed, $\Xi$ is convex as the supremum of affine functions. Furthermore, the convexity of $\dom\Theta$ follows from the L\"owner convexity of $\cf$; this immediately yields the convexity of $\Theta$.  

Then it is easy to see that $$\tilde{\mathcal I}(v_0, T)=\inf_{(z,M)\in {\mathcal Z}} \Theta(z,M)+\Xi(z,M).$$ 
In order to apply the Fenchel-Rockafellar duality theorem \cite[Theorem 1.9]{villani03topics}, we need to show that \begin{equation}\label{dualfr}\tilde{\mathcal J}(v_0, T)=\sup_{(E,B)\in {\mathcal Z}^*} -\Theta^*(-E,-B)-\Xi^*(E,B)\end{equation} and that there exists a point $(z_0,M_0)\in {\mathcal Z}$ such that $\Theta, \Xi$ are finite at $(z_0,M_0)$ and $\Theta$ is continuous at that point. 

Firstly, such a point is easy to find, it suffices to take $z_0(t,x)\equiv v_0(x)$ and $M_0(t,x) \equiv CI$ with $C$ being a constant strictly larger than the maximum in $x\in \Omega$ of the largest eigenvalue of $\cf(v_0(x))$. 

Indeed, observe first that $\Psi:=tI$ satisfies $L \Psi=0$, $\Psi(0)= 0$, and, in particular, is an admissible test function in  \eqref{e:constrweak}. Hence, for any pair of measures $(E,B)$ satisfying \eqref{e:constrweak},  we have $\la I,B\ra=0$, whence $\Xi(z_0,M_0)=\sup_{E,B:\,\eqref{e:constrweak}}\la CI,B\ra=0\in \mathbb R$. Moreover, by construction $$M_0(t,x)> \cf(z_0(t,x))$$ in the sense of symmetric matrices, and this inequality remains valid if we slightly perturb $(z_0,M_0)$ in the uniform topology (indeed,  since the image of $v_0$ is compact in $\ooo$ and $\cf$ is differentiable in $\ooo$, the map $v\mapsto \cf(v)$ is continuous at $z_0$ with respect to the uniform topology). Hence, $\Theta(z,M)=\frac 1 2 \la M,\wei I\mut \ra$ in a neighborhood of $(z_0,M_0)$, whence $\Theta$ is finite and continuous at $(z_0,M_0)$.

We now compute the Legendre transforms.  We first compute \begin{multline*}\Theta^*(-E,-B)=\sup_{(z,M)\in \mathcal Z}  -[\la z,E\ra+\la M,B\ra ]-\Theta(z,M)\\=\sup_{M\ge \cf(z)} -[\la z,E \ra+\frac 1 2 \la M,\wei I \mut+2B\ra]=-\mathcal K(E,B). \end{multline*}
We now claim that $\Xi^*(E,B)=\Upsilon(E,B)$, where $\Upsilon(E,B)$ is defined to be $\la v_0,E \ra $ if the pair of measures $(E,B)$ satisfies \eqref{e:constrweak} and is equal to $+\infty$ otherwise.  Indeed, $\Upsilon$ is convex and lower-semicontinuous, so it suffices to observe that $\Upsilon^*=\Xi$, which is an immediate consequence of the definition of the Legendre transform.

These computations imply that \eqref{dualfr} holds. By the Fenchel-Rockafellar duality, the $\sup$ in \eqref{dualfr} is achieved (in other words, \eqref{e:conc} has a maximizer) and $\tilde{\mathcal J}(v_0, T)=\tilde{\mathcal I}(v_0, T)<+\infty$. Finally, we have observed in Remark \ref{rweak} that $\tilde{\mathcal J}(v_0, T)\ge 0$.
\end{proof}

\begin{remark}\label{plusin} Assumption \eqref{e:plo} is essential since it prevents the blow-up $\tilde{\mathcal J} (v_0, T)=+\infty$, cf. \cite[Example 1]{V22}. On the other hand, \eqref{e:plo} is a very mild assumption. For example, it holds for any differential operator $L = L(M)$ without zero-order terms, and even for operators with zero-order terms depending only on the off-diagonal entries of $M$. Both options occur in our examples in Section~\ref{Sec3}.

\end{remark}

\begin{remark} Since neither the definitions of weak solutions and subsolutions nor the proof above use the conservativity assumption \eqref{e:acons}, Theorem \ref{t:exweak} is valid without imposing \eqref{e:acons}.\end{remark}

\section{Applications to compressible fluid models} \label{Sec3}

In this section, we verify the assumptions of Section \ref{setti} and the technical condition \eqref{e:plo} for the three compressible fluid models under consideration, arranged in order of increasing complexity. Consequently, all the results above --- including the existence of variational dual solutions and the Dafermos principle --- apply to these models.

\subsection{Compressible barotropic fluids} \label{s:comp} The system of  equations of motion of compressible inviscid barotropic fluids (also known as the isentropic Euler system) reads

\begin{gather} \label{cbf1} \p_t q+\di (\frac {q\otimes q}{\rho})+\nabla (P(\rho))=0,\\ \label{cbf2} 
	\p_t \rho+\di q=0,\\ \label{cbf3} 
	q(0)=q_0,\quad \rho(0)=\rho_0.\end{gather} The unknowns are $(q,\rho):[0,T]\times \Omega\to \R^d\times \R_+$, representing the mass flux and the density, respectively.
	
	\begin{assumption} \label{assp} We assume that the function $P:[0,+\infty)\to \R$ is continuous and convex, and that it is given by $$P(0)=0,\quad P(y)=U'(y)y-U(y)+U(0),\quad y> 0,$$ where $U:[0,+\infty)\to \R$ is a given continuous convex energy function such that for $y>0$ the function $U$ is smooth and $U''$ is strictly positive. In particular, this implies $P\ge 0$. We further assume that the function $2U(y)-dP(y)$ is  convex. Typical examples are \be U(y)=y^{\gamma},\quad 1<\gamma\leq 1+\frac 2 d,\ee and \be \label{ulogu} U(y)=y\log y, \ee in any dimension $d$. These examples correspond to classical physically relevant models of compressible fluids (see, for instance, \cite{CVC21,CMP18}).
		\end{assumption}

Observe that since $2U(y)-dP(y)$ is convex, there exists an affine function $l(y)$ such that $2U(y)+l(y)-dP(y)-y\geq 0$ for $y\geq 0$. Since the values of the function $P$ and the convexity properties of $U$ are not affected by adding an affine function to $U$, we may assume without loss of generality that \be 2U(y)-dP(y)-y\geq 0. \ee

\begin{remark} \label{brenres} In the pioneering work \cite{CMP18}, Brenier included the compressible Euler system \eqref{cbf1}–\eqref{cbf3} in his analysis, with particular emphasis on the energy \eqref{ulogu}. However, the only rigorous result currently available regarding the dual formulation of this system concerns consistency over small time intervals \cite[Theorem 3.1]{CMP18}.\end{remark}

In order to rewrite the problem in the abstract form \eqref{e:aeuler}--\eqref{e:aeulero}, we let $$n=N=d+1,\ Z=0,\ v=(q,\rho),\ v_0=(q_0,\rho_0), \ \lc\equiv 0,$$ 
$$\dom \cf={[\rho>0]\cup {(0,0)}},$$ $$\cf(v)=\frac {v\otimes v} \rho+\sum_{i=1}^{d} P(\rho)e_i\otimes e_i+[2U(\rho)-dP(\rho)-\rho]e_{N}\otimes e_{N}\ \textrm{if}\ \rho>0,\quad \cf(0,0)=2U(0)e_{N}\otimes e_{N}.$$
Note that the Löwner convexity and positive semidefiniteness of the first term are well known, and the Löwner convexity and positive semidefiniteness of the remaining rank-$1$ terms follow from the convexity and nonnegativity of $P$ 
and $2U(\rho)-dP(\rho)-\rho$. 

Consequently, $$\ce(v)=\ce(q,\rho)=\frac {|q|^2}{2\rho}+U(\rho)\ \textrm{if}\ \rho>0,\quad \ce(0,0)=U(0),$$
$$\ooo=[\rho>0],$$
 $$\mom=(u,\zeta)=\left(\frac q \rho, -\frac {|q|^2}{2\rho^2}+U'(\rho)\right).$$ From the perspective of physics, $u$ is the velocity of the fluid. It is clear that $\ce$ is strictly convex in $\ooo$. Hence, $v\mapsto \mom$ is injective, and the inverse is determined by the relations $$v(u,\zeta)=(u\rho,\rho),\ \rho=(U')^{-1}(\zeta+\frac 1 2|u|^2).$$ We also set $$L: D(L)\subset X_s^{\nn}\to X^n, \quad L\,\left(\begin{array}{@{}c|c@{}}
 	\Xi& \eta \\  \hline 
 	\eta^\top & \kappa
 \end{array}\right)=\left(\begin{array}{@{}c@{}}
 	-\di \Xi \\ \hline
 	-\di \eta
 \end{array}\right).$$ Then it is straightforward to compute\footnote{Hereafter, we omit $dx$ and $\mut$ in the context of dual operators to avoid heavy notation.} that $$L^*\left(\begin{array}{@{}c@{}}
\xi\\ \hline
\theta
\end{array}\right)=\frac 1 2 \,\left(\begin{array}{@{}c|c@{}}
\nabla \xi+(\nabla \xi)^\top& \nabla \theta \\  \hline 
 (\nabla \theta)^\top & 0
\end{array}\right).$$  It is easy to check that 
 \eqref{cbf1}-\eqref{cbf3} is equivalent to a particular instance of \eqref{e:aeuler}--\eqref{e:aeulero}.  Obviously, condition \eqref{e:plo} holds and $\lc$ has closed range. Moreover, the {conservativity} condition \eqref{e:acons} becomes \be\label{e:consbar}\int_\Omega \left[\frac {v\otimes v} \rho+\sum_{i=1}^{d} P(\rho)e_i\otimes e_i+[2U(\rho)-dP(\rho)-\rho]e_{N}\otimes e_{N}\right]:\left(\begin{array}{@{}c|c@{}}
 	\nabla u&  \nabla \zeta \\  \hline 
 	0 & 0
 \end{array}\right)\,dx\ee for smooth $(q,\rho): \Omega \to \ooo$.
This is equivalent to
\be\label{e:consbar1}\int_\Omega \frac {q\otimes q} \rho:\nabla u +P(\rho)\di u\,dx+q\cdot \nabla \zeta\,dx=0\ee
and hence to \be\label{e:consbar2}\int_\Omega \rho u\otimes u:\nabla u -u \cdot P'(\rho)\nabla \rho \,dx+\rho u\cdot \nabla \left(-\frac 1 2 |u|^2+U'(\rho)\right) \,dx=0,\ee which is true because $U''(\rho)\rho=P'(\rho)$. Consequently, all our assumptions are satisfied for this model, and the results above are fully applicable. In particular, the existence result and the Dafermos principle read 
\begin{corollary} \label{exbar} For any $(q_0,\rho_0)\in  C(\Omega;\R^{d+1})$, $\rho_0>0$, there exists a maximizer $$(E,B)\in  \mathcal M([0,T]\times \Omega;\R^{d+1})\times \mathcal M([0,T]\times \Omega;\R^{(d+1)\times (d+1)}_s)$$ of the dual problem \eqref{e:conc} for \eqref{cbf1}--\eqref{cbf3}, and there is no duality gap in the sense of \eqref{ndg}. \end{corollary}

\begin{corollary} \label{cordaf} Let $(q_0,\rho_0)\in  C(\Omega;\R^{d+1})$, $\rho_0>0$. Let $(q,\rho)$ be a strong solution to \eqref{cbf1}--\eqref{cbf3} with total entropy $$K(t):=\int_\Omega \left( \frac {|q(t,x)|^2}{2\rho(t,x)}+U(\rho(t,x))\right)\,d x,$$ and $(\tilde q,\tilde \rho, M)$ a subsolution with total entropy \eqref{tess}. Then, for any $0\leq t_0<t_1\leq T$, it cannot simultaneously be that $\tilde K(t)\le K(t)$ for $t\in [0,t_0]$ and $\tilde K(t)<K(t)$ for $t\in(t_0,t_1)$. In particular, it is impossible that $\tilde K(t)<K(t)$ for $t\in(0,\epsilon)$, $\epsilon>0$. \end{corollary}

\begin{remark} This ``Dafermos principle'' pertains to strong solutions $(q,\rho)$, whereas the adversaries  $(\tilde q,\tilde \rho, M)$ are subsolutions in a ``weak'' sense, so it does not contradict the results of \cite{CK14}.\end{remark}

\begin{remark} \label{sharpb}
In order to obtain the ``sharp'' formulation, we compute $$\p_l\cf(q,\rho)=\left(\begin{array}{@{}c|c@{}}
	u\otimes e_l+e_l\otimes u& e_l \\  \hline 
	e_l^\top & 0
\end{array}\right),\ l=1,\dots,d,$$
$$\p_N\cf(q,\rho)=\left(\begin{array}{@{}c|c@{}}
	-u\otimes u+P'(\rho)I& 0 \\  \hline 
	0 & 2U'(\rho)-dP'(\rho)
\end{array}\right).$$ Hence  \eqref{e:aeuler2} becomes \begin{gather} \label{e:first} \p_t u_l+ \left(\begin{array}{@{}c|c@{}}
		\nabla u& \nabla \zeta \\  \hline 
		0 & 0
	\end{array}\right):\left(\begin{array}{@{}c|c@{}}
	u\otimes e_l+e_l\otimes u& e_l \\  \hline 
	e_l^\top & 0
\end{array}\right)=0,\ l=1,\dots,d,\\
\p_t \zeta-\nabla u:u\otimes u+P'(\rho)\di u=0, \label{e:second}\\ \rho=(U')^{-1}(\zeta+\frac 1 2|u|^2). \label{third}\end{gather}  
Equation \eqref{e:first} equivalently reads
\be\p_t u+ (u\cdot\nabla)u+\frac 1 2 \nabla ( |u|^2)+\nabla \zeta=0.\label{fourth}\ee

Thus, \eqref{e:second}--\eqref{fourth} constitutes the  ``sharp'' formulation of the isentropic Euler system. 
To see that it is (at least formally) equivalent to \eqref{cbf1}--\eqref{cbf2}, we observe that \eqref{fourth} can be recast as
\be \label{e:rewr1} \p_t u+ (u\cdot\nabla)u+\nabla \left(U'(\rho)\right)=0.\ee Furthermore, since $\p_t \zeta=U''(\rho)\p_t\rho-u\p_t u$, \eqref{e:second} rewrites as \be
	U''(\rho)\p_t\rho+u \left( (u\cdot\nabla)u+\nabla \left(U'(\rho)\right)\right)-\nabla u:u\otimes u+P'(\rho)\di u=0.\label{e:rewr2} \ee
Since $U''(\rho)>0$ for $\rho>0$ and $P'(\rho)=U''(\rho)\rho$, a direct calculation shows that \eqref{e:rewr2} is equivalent to \eqref{cbf2}. Consequently, \eqref{cbf1}--\eqref{cbf2} is formally equivalent to  \eqref{e:rewr1}--\eqref{e:rewr2}, and hence to  \eqref{e:second}--\eqref{fourth}. {Observe also that the transformed system \eqref{e:second}--\eqref{fourth} is not a symmetric hyperbolic system, in line with Remark \ref{remsharp}.} \end{remark}

\subsection{Quantum fluids} The system of equations of motion of compressible quantum fluids, also known as the QHD system, is

\begin{gather} \label{nls1} \p_t q+\di (\frac {q\otimes q}{\rho})+\nabla (P(\rho))-\rho \nabla (\frac {2 \Delta \sqrt \rho}{\sqrt \rho})=0,\\ \label{nls2} 
	\p_t \rho+\di q=0,\\ \label{nls3} 
	q(0)=q_0,\quad \rho(0)=\rho_0.\end{gather} As before, the unknowns are $(q,\rho):[0,T]\times \Omega\to \R^d\times \R_+$. The function $P$ is assumed to satisfy Assumption \ref{assp}. We do not require that the initial velocity $\frac {q_0} {\rho_0}$ is irrotational. 
	
	The first equation can be written in the divergence form \be  \label{nls4} \p_t q+\di (\frac {q\otimes q}{\rho})+\nabla (P(\rho)-\Delta \rho)+\di \left(\frac {\nabla \rho\otimes \nabla \rho}{\rho} \right)=0.\ee  We define a new variable $\g:=\nabla \rho$ and express the problem in the form
		\begin{gather} \label{nls6} \p_t q+\di (\frac {q\otimes q}{\rho})+\di \left(\frac {\g\otimes \g}{\rho} \right)+\nabla (P(\rho)-\di \g)=0,\\ \label{nls7} 
	\p_t \rho+\di q=0,\\ \label{nls8} 
	\p_t \g +\nabla \di q=0,\\ \label{nls9} 
	q(0)=q_0,\quad \g(0)=\g_0:=\nabla \rho_0,\quad \rho(0)=\rho_0,\end{gather}
	with the linear constraint \be \label{nls99}  \nabla \rho-\g=0.\ee
	
	To cast this problem into the abstract framework \eqref{e:aeuler}--\eqref{e:aeulero}, we let $$n=N=2d+1,\ Z=d,\ v=(q,\g,\rho),\ v_0=(q_0,\g_0,\rho_0),$$ $$\lc (v)= \nabla \rho-\g,$$
	$$\dom \cf={[\rho>0]\cup {(0,0,0)}},$$ $$\cf(v)=\frac {v\otimes v} \rho+\sum_{i=1}^{d} P(\rho)e_i\otimes e_i+[2U(\rho)-dP(\rho)-\rho]e_{N}\otimes e_{N}\ \textrm{if}\ \rho>0,$$ $$ \cf(0,0,0)=2U(0)e_{N}\otimes e_{N},$$ $$L: D(L)\subset X_s^{\nn}\to X^n, \quad L\,\left(\begin{array}{@{}c|c|c@{}}
		\Xi & \Gamma  &  g \\  \hline 
		\Gamma^\top & \Upsilon & r \\ \hline g^\top & r^\top & \kappa 
	\end{array}\right)=\left(\begin{array}{@{}c@{}}
		-\di \Xi -\di \Upsilon+\nabla \di r\\ \hline 
		-\nabla \di g \\ \hline -\di g
	\end{array}\right).$$ Hence,
	$$\ce(v)=\ce(q,\g,\rho)=\frac {|q|^2}{2\rho}+\frac {|\g|^2}{2\rho}+U(\rho)\ \textrm{if}\ \rho>0,$$ $$ \ce(0,0,0)=U(0),$$
	$$\ooo=[\rho>0],$$
	$$\mom=(u,\et,\zeta)=\left(\frac q \rho, \frac \g \rho, -\frac {|q|^2+|\g|^2}{2\rho^2}+U'(\rho)\right).$$ Physically, $u$ corresponds to the velocity of the fluid, and $\et$ to the osmotic velocity.
	
	Note that \eqref{e:plo} holds and $\lc$ has closed range (that coincides with $X^d$). It is also straightforward to check that \begin{gather*} \lc^*(\eta)=-(0,\eta,\di \eta),\\ L^*\left(\begin{array}{@{}c@{}}
	\eta \\ \hline
	\upsilon \\ \hline \theta 
	\end{array}\right)=\frac 1 2 \,\left(\begin{array}{@{}c|c|c@{}}
	\nabla \eta+(\nabla \eta)^\top& 0 & -\nabla \di \upsilon +\nabla \theta  \\  \hline 
	0& \nabla \eta+(\nabla \eta)^\top & \nabla \di \eta \\ \hline (-\nabla \di \upsilon+\nabla \theta)^\top & (\nabla \di \eta)^\top & 0
	\end{array}\right).\end{gather*}

	The {conservativity} condition \eqref{e:acons} becomes \begin{multline} \label{e:consbarq}\int_\Omega \left[\frac {v\otimes v} \rho+\sum_{i=1}^{d} P(\rho)e_i\otimes e_i+[2U(\rho)-dP(\rho)-\rho]e_{N}\otimes e_{N}\right]\\:\left(\begin{array}{@{}c|c|c@{}}
			\nabla u& 0 & -\nabla \di \et +\nabla \zeta  \\  \hline 
			0 & \nabla u &  \nabla \di u \\ \hline 0 & 0 & 0
		\end{array}\right)\,dx=0.\end{multline}
	We omit the proof (for smooth $\rho, q, v, \eta, \zeta$ with $\rho > 0$), as it follows the same line of argument as the more general proof of \eqref{e:consbark} below. Consequently, all of our assumptions are satisfied for this model, and the preceding ``abstract'' results apply in full. In particular, a result similar to Corollary \ref{exbar} and Corollary \ref{excort} holds, as does a result analogous to Corollary \ref{cordaf}.

	\begin{remark} \label{sharpq} To pass to the ``sharp'' formulation, we compute
	$$\p_l\cf(q,\g,\rho)=\left(\begin{array}{@{}c|c|c@{}}
	u\otimes e_l+e_l\otimes u&  \et \otimes e_l & e_l  \\  \hline 
	e_l\otimes \et & 0 & 0 \\ \hline e_l ^\top & 0 & 0
	\end{array}\right),\ l=1,\dots,d,$$ 
	$$\p_{l+d}\cf(q,\g,\rho)=\left(\begin{array}{@{}c|c|c@{}}
	0&  u \otimes e_l  & 0  \\  \hline 
	e_l\otimes u & \et \otimes e_l  + e_l\otimes \et & e_l  \\ \hline 0 & e_l ^\top & 0
	\end{array}\right),\ l=1,\dots , d,$$
	$$\p_N\cf(q,\g,\rho)=\left(\begin{array}{@{}c|c|c@{}}
		-u\otimes u+P'(\rho)I& -u\otimes \et & 0 \\  \hline -\et\otimes u & -\et\otimes \et& 0\\  \hline 
		0 & 0 & 2U'(\rho)-dP'(\rho)
	\end{array}\right).$$

Remember that $\lc^*(\pi)=-(0,\pi,\di \pi)$. Hence the first $d$ components of the ``sharp'' formulation \eqref{e:aeuler2} become \begin{gather} \p_t u_l+ \left(\begin{array}{@{}c|c|c@{}}
		\nabla u & 0 & -\nabla \di \et +\nabla \zeta  \\  \hline 
		0& \nabla u & \nabla \di u \\ \hline 0 & 0 & 0
	\end{array}\right):\left(\begin{array}{@{}c|c|c@{}}
		u\otimes e_l+e_l\otimes u&  \et \otimes e_l & e_l  \\  \hline 
		e_l\otimes \et & 0 & 0 \\ \hline e_l ^\top & 0 & 0
	\end{array}\right)=0,\\ \notag l=1,\dots,d. \end{gather}
	
	 This can be rewritten in the form
	\be \p_t u+ (u\cdot\nabla)u+\frac 1 2 \nabla ( |u|^2) +\nabla \zeta -\nabla \di \et =0. \label{e:seventh}\ee
	
	
	

	The next $d$ components of the ``sharp'' formulation \eqref{e:aeuler2} are
	
	\begin{gather} \p_t \et_{l}+ \left(\begin{array}{@{}c|c|c@{}}
		\nabla u & 0 & -\nabla \di \et +\nabla \zeta  \\  \hline 
		0& \nabla u & \nabla \di u \\ \hline 0 & 0 & 0
	\end{array}\right):\left(\begin{array}{@{}c|c|c@{}}
		0&  u \otimes e_l  & 0  \\  \hline 
		e_l\otimes u & \et \otimes e_l  + e_l\otimes \et & e_l  \\ \hline 0 & e_l ^\top & 0
	\end{array}\right)+\pi_l=0,\\ \notag l=1,\dots,d. \end{gather}
	
	This can be rewritten in the form
	
	\be \p_t \et+ (\nabla u)^\top.\et+(\et\cdot \nabla) u +\nabla \di u+\pi=0.\label{e:eighth}\ee 
	Finally, the last component is  \begin{gather}  \label{e:fifth}\p_t \zeta-\nabla u:u\otimes u-\nabla u:\et\otimes \et+P'(\rho)\di u+\di \pi=0,\\ \rho=(U')^{-1}(\zeta+\frac 1 2|u|^2+\frac 1 2|\et|^2 ).\label{e:sixth}\end{gather} To make this fully consistent with the original QHD system, we note that the linear constraint \eqref{nls99} is equivalent to \be \label{ggl} \nabla \rho=\rho \et.\ee  Thus, \eqref{e:seventh},\eqref{e:eighth}--\eqref{ggl} is the ``sharp'' formulation of the QHD system \eqref{nls1}--\eqref{nls2}. It involves $3d+2$ scalar equations for a $3d+2$-dimensional unknown function $(u,\et,\zeta,\pi,\rho)$. Here $\pi:[0,T]\times \Omega\to \R^d$ is a Lagrange multiplier, playing a role analogous to the pressure in the incompressible Euler and Navier–Stokes equations, and associated with the constraint  \eqref{ggl}. \end{remark}

\subsection{Korteweg fluids} Consider now the
Euler-Korteweg system: 
\begin{gather} \label{kls1} \p_t q+\di (\frac {q\otimes q}{\rho})+\nabla (\tilde P(\rho))-\rho \nabla \left(\kp(\rho)\Delta \rho+\frac 1 2 \kp'(\rho)|\nabla \rho|^2\right)=0,\\ \label{kls2} 
	\p_t \rho+\di q=0,\\ \label{kls3} 
	q(0)=q_0,\quad \rho(0)=\rho_0.\end{gather} Similarly to the previous examples, the pair of unknowns $(q,\rho):[0,T]\times \Omega\to \R^d\times \R_+$ represents the mass flux and the density of a capillary fluid. Following \cite{BGL19}, we restrict ourselves to power-law capillary coefficients $k$ and assume that $$k(\rho)=\frac {(s+3)^2}{4} \rho^s,\quad s> -1,$$ (the case $s=-1$ corresponds to the quantum fluids and is omitted from this discussion since it was already handled above). Set $$\gamex:=\frac {s+3} 2.$$ 
	 The function $\tilde P$ is supposed to admit the decomposition \be \tilde P(\rho)=(\gamex-1)\rho^{2\gamex -1} + P(\rho)\ee with $P$ satisfying Assumption \ref{assp}. For example, $$\tilde P(\rho)=\frac{s+2} 2 y^{s+2},\ -1<s\leq -1+\frac 2 d,$$ fits here (because the corresponding $P(\rho)=\frac 1 2\rho^{s+2}=\frac 1 2\rho^{2\gamex -1}$ satisfies Assumption \ref{assp}). As above, we do not require that the initial velocity $\frac {q_0} {\rho_0}$ is irrotational.

	We now introduce the auxiliary variables $$\zp=\rho^\gamex,\ \g:=\nabla \zp.$$ 
	
	Using these variables and applying the generalized Bohm identity \cite{BGL16,BGL19}, after somewhat tedious but straightforward algebraic manipulations we can rewrite the Euler–Korteweg system in the following augmented form: \begin{gather} \notag \p_t q+\di (\frac {q\otimes q}{\rho})+\nabla \left((\gamex-1)\frac{\zp^2}{\rho} + P(\rho)\right)\\  \label{kls4}
		=\gamex \nabla \di \left(\frac{\zp \g }{\rho}\right) - \di \left(\frac{\g \otimes \g }{\rho}\right)  +(1-\gamex)\nabla \left (\frac{|\g|^2 }{\rho}\right ) , \\ \label{kls6} 
		\p_t \zp+\gamex \di \left (\frac {q \zp}{\rho} \right)+(1-\gamex) \left(\frac {q \cdot \g}{\rho}\right)=0,\\ \label{kls7} 
		\p_t \g+\gamex \nabla \di \left (\frac {q \zp}{\rho} \right)+(1-\gamex) \nabla \left(\frac {q \cdot \g}{\rho}\right)=0, \\ \label{kls8} 
		\p_t \rho+\di q=0, \\ \label{kls9} 
		q(0)=q_0,\quad \zp(0)=\zp_0:=\rho_0^\gamex, \quad \g(0)=\g_0:=\nabla \rho_0^\gamex,\quad \rho(0)=\rho_0,\end{gather}
	with the linear constraint \be \label{kls99}  \nabla \zp-\g=0.\ee
	
	\begin{remark} We deliberately ignore the nonlinear constraint $\zp=\rho^{\gamex}$ because it is already built into the system be means of \eqref{kls6} and \eqref{kls8}.  \end{remark}
	
		We let $$n=N=2d+2,\ Z=d,\ v=(q,\g,\zp,\rho),\ v_0=(q_0,\g_0,\zp_0,\rho_0),$$ $$\lc (v)= \nabla \zp-\g,$$
	$$\dom \cf={[\rho>0]\cup {(0,0,0,0)}},$$ $$\cf(v)=\frac {v\otimes v} \rho+\sum_{i=1}^{d} P(\rho)e_i\otimes e_i+[2U(\rho)-dP(\rho)-\rho]e_{N}\otimes e_{N}\ \textrm{if}\ \rho>0,$$ $$ \cf(0,0,0,0)=2U(0)e_{N}\otimes e_{N}.$$
	$$\ce(v)=\ce(q,\g,\zp,\rho)=\frac {|q|^2}{2\rho}+\frac {|\g|^2}{2\rho}+\frac {\zp^2}{2\rho}+U(\rho)\ \textrm{if}\ \rho>0,$$ $$ \ce(0,0,0,0)=U(0),$$
	$$\ooo=[\rho>0],$$
	$$\mom=(u,\et,\emm,\zeta)=\left(\frac q \rho, \frac \g \rho, \frac \zp \rho, -\frac {|q|^2+|\g|^2+\zp^2}{2\rho^2}+U'(\rho)\right).$$
	Physically, $u$ is the fluid velocity, whereas 
	$\et$ represents the so-called ``drift velocity'' \cite{BGL19}, which has been frequently used in the analysis of the Euler-Korteweg system and was referred to as the ``good'' variable in \cite{Benz,BGL16}.
	
	We also define
	$$L: D(L)\subset X_s^{\nn}\to X^n,$$ $$L\,\left(\begin{array}{@{}c|c|c|c@{}}
		\Xi & \Gamma  &  a &  g \\  \hline 
		\Gamma^\top & \Upsilon &  b & r \\ \hline a & b & \alpha & \beta \\ \hline g^\top & r^\top & \beta & \kappa 
	\end{array}\right)=\left(\begin{array}{@{}c@{}}
		-\di \Xi -\di \Upsilon-(\gamex-1)\nabla \alpha +\gamex \nabla \di b +(1-\gamex)\nabla (\tr \Upsilon)\\ \hline 
		-\gamex \nabla \di a+(\gamex-1)\nabla (\tr \Gamma) \\ \hline -\gamex \di a+ (\gamex-1)\tr \Gamma \\ \hline -\di g
	\end{array}\right).$$ Then \eqref{kls4}--\eqref{kls99} can be represented in the abstract form \eqref{e:aeuler}--\eqref{e:aeulero}. Note that \eqref{e:plo} still holds in spite of the presence of a zero-order term, and $\lc$ has closed range. Moreover, it is easy to see that \begin{multline*}L^*\left(\begin{array}{@{}c@{}}
		\eta \\ \hline
		\upsilon \\ \hline \pi  \\ \hline \theta
	\end{array}\right)\\ =\frac 1 2 \,\left(\begin{array}{@{}c|c|c|c@{}}
		\nabla \eta+(\nabla \eta)^\top& (1-\gamex)I(\di \upsilon-\pi) & -\gamex \nabla \di \upsilon +\gamex \nabla \pi & \nabla \theta \\  \hline 
		(1-\gamex)I(\di \upsilon-\pi) & \nabla \eta+(\nabla \eta)^\top +2(\gamex-1)I\di \eta& \gamex \nabla \di \eta & 0 \\ \hline (-\gamex \nabla \di \upsilon+\gamex \nabla \pi)^\top & \gamex (\nabla \di \eta)^\top & 2(\gamex -1) \di \eta & 0 \\ \hline (\nabla \theta)^\top & 0 & 0 & 0
	\end{array}\right).\end{multline*} The {conservativity} condition \eqref{e:acons} becomes \begin{multline} \label{e:consbark}\int_\Omega \left[\frac {v\otimes v} \rho+\sum_{i=1}^{d} P(\rho)e_i\otimes e_i+[2U(\rho)-dP(\rho)-\rho]e_{N}\otimes e_{N}\right]\\:\left(\begin{array}{@{}c|c|c|c@{}}
	\nabla u& (1-\gamex)I(\di \et-\emm) & -\gamex \nabla \di \et +\gamex \nabla \emm & \nabla \zeta \\  \hline 
	0 & \nabla u +(\gamex-1)I\di u& \gamex \nabla \di u & 0 \\ \hline 0 & 0 & (\gamex -1) \di u & 0 \\ \hline 0 & 0 & 0 & 0
	\end{array}\right)\,dx=0.\end{multline}
This claim is equivalent to
	\begin{multline*}\int_\Omega \frac {q\otimes q} \rho:\nabla u +P(\rho)\di u\,dx+q\cdot \nabla \zeta\,dx \\ + \frac {\g\otimes \g} \rho:(\nabla u+(\gamex-1)I\di u) +(\gamex -1)  \frac {\zp^2} \rho\di u\\+ (1-\gamex)(\di \et-\emm)\frac {q\cdot \g}\rho+(-\gamex \nabla \di \et +\gamex \nabla \emm) \cdot \frac {q\zp}\rho+\gamex \nabla \di u \cdot \frac {\g\zp}\rho=0\end{multline*}
	and hence to \begin{multline*}I_1+I_2+I_3+I_4:=\\ \int_\Omega \rho u\otimes u:\nabla u -u \cdot P'(\rho)\nabla \rho +\rho u\cdot \nabla \left(-\frac 1 2 |u|^2+U'(\rho)\right) \,dx\\+\int_\Omega\left(\rho u\cdot \nabla \left(-\frac 1 2 |\et|^2\right)+\rho \et\otimes \et:\nabla u -\nabla \di \et \cdot \rho \emm u + \nabla \di u \cdot \rho \emm \et \right)\,dx\\+(\gamex-1)\int_\Omega \left(\rho |\et|^2  \di u-(\di \et) u \cdot \rho \et -\nabla \di \et \cdot \rho \emm u + \nabla \di u \cdot \rho \emm \et \right)\,dx\\ +\int_\Omega\left(\rho u\cdot \nabla \left(-\frac 1 2 \emm^2\right)+ (\gamex -1)  \rho \emm^2 \di u+(\gamex-1)\emm u \cdot \rho \et +\gamex \rho \emm \nabla \emm \cdot u\right)\,dx =0.\end{multline*} The first integral $I_1$ is zero as explained in Section \ref{s:comp}.  Observe that \be \label{gwl} \nabla(\rho \emm)=\rho \et\ee because of  \eqref{kls99}. Hence, the fourth integral can be rewritten as $$I_4=(\gamex -1)\int_\Omega\di (\rho \emm^2 u)\,dx =0.$$  Integrating by parts and leveraging \eqref{gwl}, we get \begin{multline*} I_2=\int_\Omega\Big[-\rho \et\otimes u:\nabla \et+\rho \et\otimes \et:\nabla u +\nabla \et : \rho \emm (\nabla u)^\top + \nabla \et: \nabla (\rho \emm)\otimes u\\ - \nabla u:  \nabla (\rho\emm)\otimes  \et - \nabla u:  (\rho\emm) (\nabla \et)^\top \Big]\,dx=0.\end{multline*}  Finally, again due to \eqref{gwl}, \begin{multline*}  {I_3}=(\gamex-1)\int_\Omega\Big[\rho |\et|^2  \di u-(\di \et) u \cdot \rho \et+(\di \et) \rho \emm \di u + (\di \et) \nabla (\rho \emm)\cdot u\\ - (\di u) \rho \emm \di \et - (\di u) \nabla (\rho \emm)\cdot \et \Big]\,dx=0.\end{multline*} This completes the proof of  \eqref{e:consbark}. Thus, all required assumptions hold for this model, and the above ``abstract'' results apply. In particular, the existence result states:
	\begin{corollary} \label{excort} For any $(q_0,\g_0,\zp_0,\rho_0)\in  C(\Omega;\R^{2d+2})$, $\rho_0>0$, $\zp_0\equiv \rho_0^\gamex$, $\g_0 \equiv \nabla \zp_0$, there exists a maximizer $$(E,B)\in  \mathcal M([0,T]\times \Omega;\R^{2d+2})\times \mathcal M([0,T]\times \Omega;\R^{(2d+2)\times (2d+2)}_s)$$ of the corresponding dual problem \eqref{e:conc} for \eqref{kls4}--\eqref{kls99}, and there is no duality gap in the sense of \eqref{ndg}. \end{corollary}
	We leave the derivation of the ``sharp'' formulation for this model as an exercise to the reader.
	
	\section{Inviscid Burgers' equation: revisiting Brenier's shock-free substitute}
	
	Let us finish our discussion  by shedding some more light on the duality features of the inviscid Burgers’ equation 
	\label{s:bur}
	\be \label{burg} \p_t v+ \frac 12 \p_x(v^2),\quad v(0)=v_0,\quad \int_{\mathbb T^1} v(x,t)\,dx=0,\ee  where $v:[0,T]\times \mathbb T^1\to \R.$
	
	This trivially fits into our abstract framework \eqref{e:aeuler}--\eqref{e:aeulero} by letting $$d=n=N=1,\ Z=0, \ \lc(v)=\int_\Omega v(x)\,dx,$$ 
	$$\ \cf(v)=v^2, \ \ooo=\dom \cf=\R,\ \ce(v)=\frac 1 2 v^2,$$
	$$\mom=v,$$
	$$L=-\frac 1 2 \p_x, \ L^*=\frac 1 2 \p_x.$$
	It is clear that the assumptions of Section \ref{setti} as well as condition \eqref{e:plo} hold. Hence, all our results are applicable. 
	
	In particular, by Theorem \ref{t:exweak}, for every $v_0\in C(\Omega)$ with zero mean there exists a variational dual solution $(E,B)\in \mathcal M([0,T]\times \Omega;\R^2)$. We stress that this existence result is not new for this model, as an analogous statement was already established in the context of Burgers’ equation (with $\wei \equiv 1$) by Brenier, see \cite[Proposition 4.4]{CMP18}.  To emphasize the link with one-dimensional mean-field games, he works with the variables $(q,\rho)$ that correspond  in our notation to $(-E,\mut+2B)$.  Moreover, it follows from \cite[Appendix B]{V25} that the corresponding dual problem (with $\wei\equiv 1$, cf. \cite[problem (B.5-B.6)]{V25}) represents a \emph{ballistic} variant of Monge–Kantorovich optimal transport in one dimension.
	
	Assume now that $v_0$ is smooth and has zero mean. In this case, Brenier shows \cite[Theorem 4.2]{CMP18} that the Radon measure $q$ is absolutely continuous with respect to $\rho$, and the Radon-Nikodym derivative $\frac {dq}{d\rho}$ can be characterized as follows. Let $v$ be the unique entropy solution to \eqref{burg}, possibly with shocks. Then he explicitly constructs an initial datum $v_0^T(x)$ such that the corresponding entropy solution $v^T(t,x)$ is bounded and Borel measurable on $[0,T]\times \Omega$, continuous on $[0,T)\times \Omega$ and satisfies \be v^T(T,x)=v(T,x),\ \mathrm{for}\ \mathrm{a.e.}\ x\in \Omega.\ee He calls $v^T$ a \emph{shock-free substitute} of $v$, and shows that \be \label{qvtr} d q=v^T d\rho.\ee
	
	Let us recall the way $v^T_0$ is constructed. Since the spatial mean of $v_0$ is zero, there exists a smooth $1$-periodic  function $\varphi:\R\to \R$ such that $\varphi'=v_0$. Let $\varphi^T:\R\to \R$ be such that $\frac {x^2}{2}+ T\varphi^T(x)$ is the convex envelope of $\frac {x^2}{2}+ T\varphi(x)$. It is easy to see that $\varphi^T$ is $1$-periodic and the contact set $\omega:=[\varphi^T=\varphi]$ is closed and  $1$-periodic. Then $$v^T_0:=(\varphi^T)'.$$ Note that $v^T_0$ is Lipschitz by \cite[Theorem 4.2]{GR90}. Actually, $v^T_0$ coincides with $v_0$ in $\omega$, and {$x+Tv^T_0(x)$} is constant on each connected component of $\Omega\backslash\omega$. Consequently, $v^T$ is locally Lipschitz on $[0,T)\times \Omega$.
	
It can also be deduced from Brenier's argument, cf.  \cite[p. 601]{CMP18}, that the support of the measure $\rho$ has gaps and the discrepancy between $v$ and $v^T$ only occurs inside those gaps. Hence, we also have  \be \label{consist} d q=v d\rho,\ee
	which shows that no anomaly occurs: \textbf{Brenier’s duality formulation is  compatible with all entropy solutions, including those containing shocks}. Moreover, we claim that the \textbf{shock-free substitute can be fully recovered from the variational dual solution 
		$(q,\rho)$, even in regions where $\rho$ vanishes}.

	More precisely, $v^T$ can be obtained from $q=-E$ via formula \eqref{e:ta1} applied in the context of Burgers' equation (were are assuming $\wei\equiv 1$ as Brenier does, though it does not seem to be necessary). Accordingly, our claim is \begin{proposition} Brenier's shock free substitute $v^T$ and the optimal flux $q$ from \cite[Theorem 4.2]{CMP18} satisfy \be \label{e:ta2} \la\psi, v^T\mut \ra= \la\int_0^t\frac {\psi(s,\cdot)}{T-s}\,ds, q\ra,\ee for any smooth test function $\psi(t,x)$ that vanishes in a neighborhood of $\{T\}\times \Omega$. \end{proposition} Note that \eqref{e:ta2}  unambiguously determines $v^T$ in $[0,T)\times \Omega$. 
	
	\begin{proof} We begin with claiming that \be \label{crux} \rho=\rho^T:=\mut(1+(T-t)\p_x(v^T))\ee in the sense of distributions on $(0,T)\times \Omega$. 
		
		To see this, observe first that, by construction, both $\rho$ and $\rho^T$ are weakly continuous functions of $t$ valued in the space of distributions on $\Omega$.  Moreover, both of them satisfy the initial condition \be \label{i:cond} \rho(0)=\rho^T(0)=\mut(1+T\p_x(v_0^T)).\ee Indeed, it follows from Brenier's argument that the support of $\rho(0)$ is contained in $\omega$, and $\rho(0)$ coincides with $\mut(1+Tv_0')$ in $\omega$. The same claims are true for $\rho^T(0)=\mut(1+T(v_0^T)')$. To see this, note first that {$x+Tv^T_0(x)$}  is constant on each connected component of $\Omega\backslash\omega$, thus the support of $\rho^T(0)$ is contained in $\omega$. Since $v^T_0=v_0$ in $\omega$, it is rather obvious that $(v_0-v_0^T)'(x)=0$ for every non-isolated point of $\omega$ where $v_0^T$ is differentiable (hence, a.e. in $\omega$). 
		
		By \cite[display after (4.60)]{CMP18} we have \be \label{e:cont} \p_t \rho=-\p_x q \ee in the sense of distributions.  Thus, due to \eqref{qvtr}, \be \label{e:cont1} \p_t \rho=-\p_x (\rho v^T)\ee in the sense of distributions on $(0,T)\times \Omega$. At the same time, we note that \be \label{e:cont2} \p_t \rho^T=-\p_x (\rho^T v^T)\ee in the sense of distributions.  Indeed, since $v^T$ is an entropy solution, we have  \begin{multline} \label{vdir} \p_t((T-t)v^T)=-(v^T+(t-T)\p_t(v^T))\\=-\left(v^T+\frac 1 2 (T-t)\p_x((v^T)^2)\right)=-\left(v^T+(T-t)v^T\p_x(v^T)\right)=-\rho^T v^T\end{multline} in the sense of distributions. Hence, $$ \p_t \rho^T=\p_t((T-t)\p_x(v^T))=\p_x\p_t((T-t)v^T)=-\p_x (\rho^T v^T)$$ in the sense of distributions.
		
		The locally Lipschitz regularity of $v^T$ guarantees uniqueness of the solutions to the Cauchy problem \eqref{e:cont1}, \eqref{i:cond}; thus $\rho=\rho^T$ as claimed.  
		
		Putting everything together, we deduce that \begin{multline*} \la\int_0^t\frac {\psi(s,\cdot)}{T-s}\,ds, q\ra = \la\int_0^t\frac {\psi(s,\cdot)}{T-s}\,ds, \rho v^T\ra = \la\int_0^t\frac {\psi(s,\cdot)}{T-s}\,ds, \rho^T v^T\ra \\ = -\la\int_0^t\frac {\psi(s,\cdot)}{T-s}\,ds, \p_t((T-t)v^T)\ra=\la \frac{\psi}{T-t}, (T-t)v^T\mut\ra =   \la \psi, v^T \mut\ra\end{multline*} for any smooth test function $\psi(t,x)$ that vanishes in a neighborhood of $\{T\}\times \Omega$.
		\end{proof}
	
	\begin{remark}  \label{nodualitygap} Let us verify that, in the context of this section, there is no gap between $\mathcal I(v_0, T)$ and $\tilde{\mathcal J}(v_0, T)$. By \eqref{ndg}, $\tilde{\mathcal I}(v_0, T)=\tilde{\mathcal J}(v_0, T)$, and it follows from \cite[Proposition 4.4]{CMP18} that $\tilde{\tilde{\mathcal I}}(v_0, T)=\tilde{\mathcal J}(v_0, T)$, where $\tilde{\tilde{\mathcal I}}$ is defined as in \eqref{e:sadd1sub}, but restricted to smooth pairs $(v,M)$. Thus it is enough to show that $\mathcal I(v_0, T)\le \tilde{\tilde{\mathcal I}}(v_0, T)$. This follows from the observation that the time integral of the total entropy $K(t)$ of the entropy solution $v$ cannot be larger than the time integral of the total entropy $\tilde K(t)$ of any smooth subsolution $(u,M)$ to \eqref{burg}. Indeed, for every such subsolution there exists a smooth function $\theta:[0,T]\times \mathbb T^1\to \R$ such that $$\p_x\theta=u,\ \p_t\theta+\frac 1 2 M=0,\ \theta(0,\cdot)=\varphi,\ M\ge u^2.$$  By comparison principle, $\theta \leq \psi$, where $\psi$ is the viscosity solution to the quadratic Hamilton-Jacobi equation emanating from $\varphi$. Consequently, \begin{multline*}\int_0^T K(t)\,dt=\frac 1 2\int_0^T (v,v)\,dt=\frac 1 2\int_0^T (\p_x \psi,\p_x \psi)\,dt \int_{\mathbb T^1}=\int_{\mathbb T^1}\left[ \varphi(x)-\psi(T,x)\right]\,dx\\ \le \int_{\mathbb T^1}\left[ \varphi(x)-\theta(T,x)\right]\,dx=\int_0^T \tilde K(t)\,dt.\end{multline*} \end{remark}

	\label{Burg}
	
	\appendix

\section{Formal derivation of \eqref{e:aeuler2} and \eqref{e:entc}.} \label{apb}

Assume \eqref{e:acons}. Let $a: \Omega \to \R^n$ be a smooth test function satisfying $\lc a=0$ and $v\in C([0,T]\times \Omega; \ooo)$ be a smooth solution to \eqref{e:aeuler}--\eqref{e:aeulero}. Then there exists a small uniform $\varepsilon>0$ so that $v(t,x)+sa(x)\in\ooo$ for $|s|<\varepsilon$. Hence, due to \eqref{e:acons}, \be  (\cf(v+sa), L^*(\nabla \ce(v+sa)))=0\ee for every $t\in[0,T]$. Taking the derivative with respect to $s$ at $s=0$ yields \be \label{o:aeulepp} \sum_{l=1}^n \left[(\p_l\cf(v)a_l, L^*(\mom))+(L(\cf(v)),\p_{v_l}(\nabla \ce (v))a_l)\right]=0.\ee
Using \eqref{e:aeuler}, we recast \eqref{o:aeulepp} in the form 
\be \label{o:aeulepp1} \sum_{l=1}^n (\p_l\cf(v):L^*(\mom),a_l)+\sum_{l,m=1}^n(\p_t v_m,\p_{lm}\ce (v)a_l)=0.\ee
By the chain rule, $\p_t (\mom)_l=\p_t (\p_l \ce(v))=\sum_{m=1}^n \p_{lm}\ce (v)\p_t v_m$. Thus \eqref{o:aeulepp1} becomes \be \label{o:aeulepp2} \sum_{l=1}^n (\p_l\cf(v):L^*(\mom)+\p_t (\mom)_l,a_l)=0.\ee Let $\zeta$ be defined by $$\zeta_l:=\p_t (\mom\mut)_l+L^* (\mom\mut):\partial_l \cf(\nabla \ce^*(\mom))=\p_t (\mom\mut)_l+L^* (\mom\mut):\partial_l \cf(v).$$ By \eqref{o:aeulepp2}, $\zeta(t)$ lies in the annihilator of $\ker \lc$. By the closed range theorem, $\lc^*$ has closed range. Consequently, $\zeta(t) \in \operatorname{ran} \lc^*$, which yields \eqref{e:aeuler2}.

In order to formally establish \eqref{e:entc}, it suffices to observe that
$$\frac d {dt} K(t)=\int_\Omega \p_t\ce (v(t,x))\,d x=(\nabla \ce(v), \p_t v)=(\mom,L(\cf(v)))\\=(L^*(\mom),\cf(v))=0.$$

	\section{Some open problems} \label{apa}

\begin{openproblem} \label{op1} It is unclear how to relax the assumption $v_0\in  C(\Omega;\ooo)$ in the setting of Theorem \ref{t:exweak}, or at least how to allow $v_0$ to take some values in $\dom \cf \backslash \ooo$. Consequently, it remains an open problem to establish the existence of variational dual solutions for the compressible fluid models considered in Section \ref{Sec3} without assuming that the initial density $\rho_0$ is continuous and strictly positive.   \end{openproblem}

\begin{remark} For models that fit into our framework with $\dom \cf = \mathbb{R}^n$, satisfy the so-called \emph{trace condition}, and whose entropy function $\ce$ is either quadratic or generates an anisotropic Orlicz space, it is typically possible to accommodate discontinuous initial data (cf. \cite{CMP18,V22,V25,AV25}). Note that the trace condition fails for the inviscid Burgers equation, see \cite[Remark 5.3]{V25}. \end{remark}

\begin{openproblem} \label{r:time} Assume \eqref{e:acons},  \eqref{e:plo}, and let $v_0\in  C(\Omega;\ooo)\cap\ker \lc$.  Let $(E,B)$ be any maximizer of \eqref{e:conc} (they exist by Theorem \ref{t:exweak}). The fact that formula \eqref{e:ta1} recovers strong solutions to \eqref{e:aeuler}--\eqref{e:aeulero} at all times, as well as discontinuous entropy solutions of the Burgers equation at the terminal time $T$ and on the support of the measure $\rho$ from Section \ref{s:bur}, suggests defining a ``generalized'' solution $\mom$ (a rough analogue of Brenier's shock-free substitute, see also \cite[Remark 4.4]{V22}, \cite[Remark 5.5]{V25}) to \eqref{e:aeuler}--\eqref{e:aeulero}  through the duality relation
	\be \label{e:ta3} \la\psi, \mom\mut \ra= -\la\int_0^t\frac {\psi(s,\cdot)}{\Wei(s)}\,ds, E\ra,\ee for any smooth test function $\psi(t,x)$ that vanishes in a neighborhood of $\{T\}\times \Omega$ and satisfies $\lc \psi=0$.
	It remains unclear under which additional assumptions on $L, \cf$ and $\lc$ one can recover 
	$v$ from 
	$\mom$ in any meaningful sense compatible with the constraints, and what the precise regularity classes of such objects 
	$v$ and $\mom$ should be.  \end{openproblem}
	
	\begin{remark} For $\lc \equiv 0$ and $\dom \cf =\R^n$ (i.e., when there are no constraints) and under the trace condition, this issue is settled in the quadratic case $\cf(v)=v \otimes v$, see \cite[Remark 4.4]{V22}, and in the ``Orlicz'' case, see \cite[Remark 5.5]{V25}. 
		\end{remark}
	 \begin{openproblem} \label{op33} The following compatibility question remains an open problem for both quadratic and non-quadratic nonlinearities, with the Burgers case having been settled above. Suppose that $v_0(x)$ is such that there exists a unique ``physically relevant'' solution $\tilde v(t,x)$ to \eqref{e:aeuler}--\eqref{e:aeulero}, in any appropriate sense, on $[0,T]\times \Omega$, and assume that \eqref{e:pd++} is violated. A natural question is then: under which structural hypotheses does the corresponding $\mom$, obtained from a variational dual solution $(E,B)$ through \eqref{e:ta3}, coincide with $\tilde{v}^{\#}$ at the terminal time $T$? Should this property hold for a certain model, the physical solution $\tilde v$ could be retrieved at all times via the “mountain climbing” strategy proposed in \cite[p.~597]{CMP18}.\end{openproblem}
	 \begin{openproblem} A related open consistency problem, likewise settled only in the Burgers case (cf. \eqref{consist}), is the following. Let $v_0$ and $\tilde v$ be as in Open Problem \ref{op33}, and suppose that \eqref{e:pd++} fails. For which structural hypotheses do the associated variational dual solutions $(E,B)$ satisfy \be \label{e:vret3} \frac 1 2 (\wei I\mut+2B):\partial_l\cf(\tilde v)+(E)_l=0,\quad l=1,\dots,n,\ee cf. \eqref{e:vret1}, in the sense of Radon measures? In particular, in the quadratic case $\cf(v)=v \otimes v$ the conjectured relation \eqref{e:vret3} has something of the flavour of  matrix-valued optimal transport (cf. \cite{BV18,Lizou25,V22}), namely \be \label{e:vret4} (\wei I\mut+2B).\tilde v=-E\ee as vector-valued Radon measures, with $(\wei I\mut+2B)$ playing the role of a ``density'', $\tilde v$ that of a ``velocity'', and $(-E)$ that of the corresponding ``mass flux''.
	 \end{openproblem}

	\subsection*{Acknowledgement} 
	The work was partially supported by CMUC\footnote{https://doi.org/10.54499/UID/00324/2025} under FCT, grants UID/00324/2025 and UID/PRR/00324/2025, and by KAUST Research Funding (KRF), award ORFS-2024-CRG12-6430.3. {The author is very grateful to the anonymous reviewer for pointing out that the change of variables \eqref{sharp} was pioneered in Godunov's work in the 1960s}, and to L.-P. Chaintron for helpful discussions.

	\end{document}